\documentclass[a4paper,11pt]{article}

\usepackage{arxiv}

\usepackage{lipsum}
\usepackage{amsfonts}
\usepackage{amsthm}
\usepackage{booktabs}
\usepackage{graphicx}
\usepackage{epstopdf}
\usepackage{amsmath}
\usepackage{moreverb}
\usepackage{amsfonts}
\usepackage{amssymb}
\usepackage{bm}
\usepackage{epsfig}
\usepackage{listings}
\usepackage{mathtools} % provides \colonequals
\usepackage{siunitx}
\usepackage{url}
\usepackage{hyperref}
\usepackage{cleveref}
\usepackage{xcolor}

\DeclareMathAlphabet{\mathpzc}{OT1}{pzc}{m}{it}

\usepackage[
	backend=biber,
	natbib=true,
	style=numeric-comp,
	citestyle=ieee,   % Citation style
	sorting=none,     % All entries are processed in citation order.
	sortcites=true,   % If multiple entries are given in one command, they will be sorted according to the option "sorting"
	doi=false,        % omit field "doi"
	isbn=false,       % omit field "isbn"
	url=false,        % omit field "url"
	eprint=false,     % omit field "eprint"
	maxbibnames=99,   % don't skip authors with "et. al."
	giveninits=true,  % Abbreviate first names
	]{biblatex}
\AtEveryBibitem{%
	\clearfield{note}%
}
\usepackage{tikz}
\usepackage{tkz-euclide}
\usetikzlibrary{
	calc,
	patterns,
	angles,
	quotes,
	positioning,
	shapes.geometric
}

\usepackage{pgfplots}
\pgfplotsset{compat=1.11}
\newcommand{\bigO}{\mathcal{O}}
\newcommand{\Z}{\mathbb{Z}}

\newcommand{\oa}{\ensuremath{\overline{\alpha}}}

\newcommand{\os}{\ensuremath{\overline{\sigma}}}

\newcommand{\uv}{\ensuremath{\underline{v}}}
\newcommand{\ov}{\ensuremath{\overline{v}}}

\newcommand{\oc}{\ensuremath{\overline{c}}}

\newcommand{\dx}{\ensuremath{\delta\xi}}
\newcommand{\dxt}{\ensuremath{\delta\xi_\tau}}

\newcommand{\xt}{\ensuremath{\xi_\tau}}
\newcommand{\nxt}{\ensuremath{\|\xi_\tau\|}}

\newcommand{\ndx}{\ensuremath{\|\dx\|}}
\newcommand{\ndxt}{\ensuremath{\|\dxt\|}}

\definecolor{darkgreen}  {RGB}{ 72, 117,  73}
\definecolor{lightgreen} {RGB}{171, 210, 130}

\definecolor{blue}       {RGB}{  3, 124, 135}
\definecolor{steelblue}  {RGB}{ 70, 130, 180}
\definecolor{darkblue}   {RGB}{ 42,  78, 108}

\definecolor{black}      {RGB}{ 16,  32,  32}
\definecolor{darkgray}   {RGB}{ 80,  80,  80}
\definecolor{lightgray}  {RGB}{167, 181, 183}

\definecolor{zibblue}    {RGB}{ 93, 188, 210}

\definecolor{firebrick}  {RGB}{178,  34,  34}
\definecolor{forestgreen}{RGB}{ 34, 139,  34}

\ifpdf
	\DeclareGraphicsExtensions{.eps,.pdf,.png,.jpg}
\else
	\DeclareGraphicsExtensions{.eps}
\fi

\newcommand{\orcid}[1]{\href{https://orcid.org/#1}{\includegraphics[height=8pt]{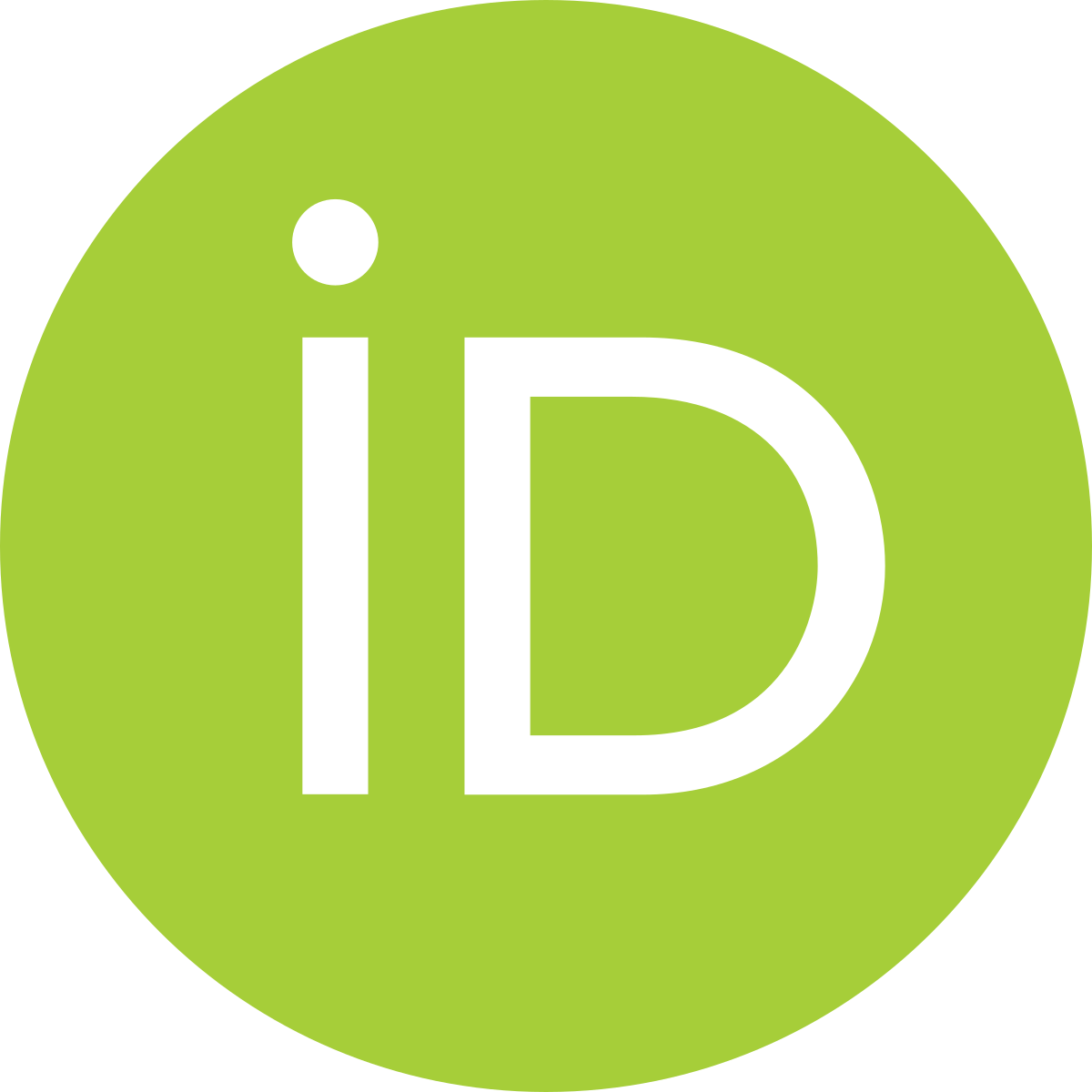}}}

\theoremstyle{thmstyleone}%
	\newtheorem{theorem}{Theorem}

	\newtheorem{lemma}[theorem]{Lemma}

\theoremstyle{thmstyletwo}%
	\newtheorem{example}{Example}
		
	\newtheorem{remark}{Remark}

\theoremstyle{thmstylethree}%
	\newtheorem{definition}{Definition}

\title{Error Bounds for Discrete-Continuous Shortest Path Problems with Application to Free Flight Trajectory Optimization}
\renewcommand{\shorttitle}{Error Bounds for Discrete-Continuous Shortest Path Problems}

\author{
	\href{https://orcid.org/0000-0001-7223-9174}{\includegraphics[scale=0.06]{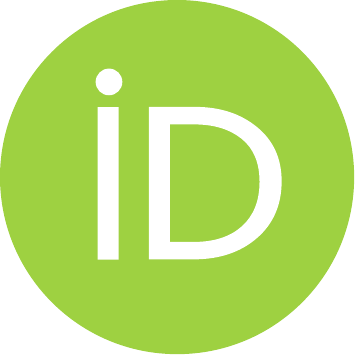}\hspace{1mm}Ralf Borndörfer}%
	\thanks{Zuse Institute Berlin, Takustr. 7, 14195 Berlin}
	\\
	\texttt{borndoerfer@zib.de} \\
	\And
	\href{https://orcid.org/0000-0002-8953-808X}{\includegraphics[scale=0.06]{orcid.pdf}\hspace{1mm}Fabian Danecker}$^*$\\
	\texttt{danecker@zib.de} \\
	\And
	\href{https://orcid.org/0000-0002-1071-0044}{\includegraphics[scale=0.06]{orcid.pdf}\hspace{1mm}Martin Weiser}$^*$\\
	\texttt{weiser@zib.de}
}

\hypersetup{
pdftitle={Error Bounds for Discrete-Continuous Shortest Path Problems with Application to Free Flight Trajectory Optimization},
pdfsubject={math.OC},
pdfauthor={Ralf Borndörfer, Fabian Danecker, Martin Weiser},
pdfkeywords={
	shortest path,
	flight planning,
	free flight,
	discretization error bounds,
	optimal control,
	discrete optimization
	}
}

\begin{document}

\maketitle

\begin{abstract}
	Two-stage methods addressing continuous shortest path problems start local minimization from discrete shortest paths in a spatial graph. The convergence of such hybrid methods to global minimizers hinges on the discretization error induced by restricting the discrete global optimization to the graph, with corresponding implications on choosing an appropriate graph density.
	\\
	A prime example is flight planning, i.e., the computation of optimal routes in view of flight time and fuel consumption under given weather conditions. Highly efficient discrete shortest path algorithms exist and can be used directly for computing starting points for locally convergent optimal control methods.
	\\
	We derive a priori and localized error bounds for the flight time of discrete paths relative to the optimal continuous trajectory, in terms of the graph density and the given wind field. These bounds allow designing graphs with an optimal local connectivity structure.
	\\
	The properties of the bounds are illustrated on a set of benchmark problems. It turns out that localization improves the error bound by four orders of magnitude, but still leaves ample opportunities for tighter error bounds by a posteriori estimators.
\end{abstract}



\section{Introduction}
There are applications of numerical optimization that call for the
computation of global instead of just local optima. One
example is free flight planning, an instance of airborne navigation,
where travel time is optimized subject to a given wind field (the
travel time $T$ between origin and destination is almost proportional
to fuel consumption, CO$_2$ emission, and cost~\cite{WellsEtAl2021}).  Going left or right
around obstacles or adverse wind situations gives rise to locally optimal trajectories with
considerably different costs~\cite{Schienle2018}, and airlines are naturally
interested in the best of those.

Various approaches to global optimization have been proposed:
stochastic ones like multistart or simulated
annealing~\cite{BoenderRomeijn1995}, biologically inspired
metaheuristics like genetic algorithms or particle
swarms~\cite{Lones2014}, and rigorous ones based on objective bounds
and branching~\cite{BelottiEtAl2013,FloudasGounaris2009}. The former
approaches converge to a global minimizer only almost surely at
increasing computational costs, but provide no guarantees for
finiteness. The latter ones usually require in-depth structural
knowledge of the objective, or the use of interval arithmetics, and
quickly suffer from the curse of dimensionality for practically relevant problems.

We consider in this paper a two-stage multistart approach along these
lines. It (i) defines a sufficiently large set of possible starting
points, (ii) selects few promising candidates, and (iii) performs
local optimization starting from those candidates. For this to be
computationally feasible, the representation and selection of starting
points needs to be highly efficient even for large and
high-dimensional design spaces. This is, of course, problem-dependent.
Some problems allow a discretization in terms of discrete network
optimization problems such as minimum cost flow and, in particular,
shortest path problems, which can be solved efficiently to global
optimality in theory and practice~\cite{KnudsenEtAl2018}. If such discrete
problems are close to their continuous counterpart, their solutions might
provide promising starting points for local optimization to converge
to a nearby global optimizer.

Obviously, flight planning and discrete shortest
path search are related in this way and can hence serve as examples to
substantiate the general idea. The starting
points covering the design space of trajectories between origin and
destination can be implicitly described as paths in a graph covering
the spatial domain. In this discrete approximation of the problem,
the selection of promising candidate points can be efficiently
performed using Dijkstra's algorithm or its $A^*$ variants. This leads
to a hybrid discrete-continuous algorithm combining discrete global
optimization methods with continuous local optimal control methods.

A first successful step in this direction has been taken with the development of the hybrid algorithm DisCOptER~\cite{BorndoerferDaneckerWeiser2021} that has been proposed by the authors of this paper for free flight planning. Note that the discrete stage alone is traditionally used as a standalone optimizer for practical flight planning on given airway networks~\cite{BlancoEtAl2016, BlancoEtAl2017}, but gets quickly inefficient when the airway networks need to be refined significantly to exploit the potential benefits of free flight~\cite{WellsEtAl2021}. Similarly, for robot path planning, rapidly exploring random graphs and trees (RRT) are used for sampling the trajectory design space at many discrete points~\cite{YangEtAl2016}.

Guaranteed convergence of a hybrid two-stage algorithm to a global
minimizer hinges on the one hand on a sufficiently dense sampling of
possible starting points in the design space, and on the other hand on
the ability of the local optimizer to converge reliably to a nearby
local optimum when started from one of these candidate points. The
present paper investigates the first aspect, i.e., we derive bounds on
the required resolution of the discretization. To this purpose, we
introduce a continuous problem formulation that allows a direct
comparison of continuous and discrete 2D flight paths
(\Cref{sec:formulation}), and derive bounds for the flight
duration deviation $T(\xi)-T(\xi_C)$ between different paths $\xi$ and
$\xi_C$ in terms of spatial distance $\|\xi-\xi_C\|$, angular distance
$\|(\xi-\xi_C)_\tau\|$, and bounds on the stationary wind and its
derivatives. Based on the $(h,l)$ graph density property
from~\cite{BorndoerferDaneckerWeiser2021}, we obtain corresponding
flight duration bounds for discrete optimal trajectories
(\Cref{sec:error-bounds}), which also yield a theoretically
optimal ratio $h=\mathcal{O}(l^2)$ of vertex distance and
characteristic edge length. We derive two types of error bounds: an a
priori bound $T(\xi_G)\le T(\xi_C) + \kappa l^2$ depending only on
problem quantities but not on a particular solution, and a local bound
based on bounds for the wind in a neighborhood of the optimal
trajectory. Taking more detailed information into account, the latter
one improves on the former one by several orders of magnitude. The
theoretical predictions are confirmed by numerical examples for a set
of benchmark problems with varying wind complexity
(\Cref{sec:numeric}), which reveal that there is still ample room
for improvement by using a posteriori error estimators. The flight
planning application leaves its imprint on the nature and derivation of
these bounds, but the general idea should work for similar applications
that have a discrete-continuous nature.

\section{Shortest flight planning:
	continuous \& discrete} \label{sec:formulation}
For simplicity of presentation, we consider flight planning in the Euclidean plane. We aim at minimizing the travel time $T$ between an origin $x_O$ and a destination $x_D$, with a fixed departure at $t=0$ and a constant airspeed $\ov > 0$, thus neglecting start and landing phase. Moreover we assume a spatially heterogeneous, twice continuously differentiable wind field $w$ to be given, with a bounded magnitude $\|w\|_{L^\infty(\mathcal{R}^2)} < \ov$. Focusing on free flight areas, we also neglect any traffic flight restrictions.

\subsection{Continuous: optimal control}
In free flight areas, the flight trajectory is not restricted to a predefined set of airways. Instead, we consider any Lipschitz-continuous path $x:[0,T]\to\mathcal{R}^2$ with $\|x_t-w\|=\ov$ almost everywhere, connecting origin $x_O$ and destination $x_D$, as a valid trajectory. Among those, we need to find one with minimal flight duration $T$, since that is essentially proportional to fuel consumption~\cite{WellsEtAl2021}. This classic of optimal control is known as Zermelo's navigation problem~\cite{Zermelo1931}.

In order to formulate the problem over a fixed interval $[0,1]$ independent of the actual flight duration, we scale time by $T^{-1}$ and end up with the optimal control problem for the flight duration $T\in\mathcal{R}$, the position $x\in H^1([0,1])$, the airspeed vector $v\in L^2([0,1])$,
\begin{equation}\label{eq:continuous-problem}
	\min_{T,x,v} T \quad \text{s.t.}\quad
	c(T,x,v) = \begin{bmatrix}
	x(0) - x_O \\
	x(1) - x_D \\
	\dot x(\tau) -T(v(\tau) +  w(x(\tau),\tau)) \\
	v(\tau)^T v(\tau) - \ov^2
	\end{bmatrix} = 0
\end{equation}
with $c:\mathcal{R}\times H^1([0,1])^2 \times L^2([0,1])^2 \to \mathcal{R}^2 \times \mathcal{R}^2 \times L^2([0,1])^2 \times L^2([0,1])$. Note that due to $v^Tv = \ov^2$, the airspeed  $v$ (and therefore also the ground speed $v+w$) is bounded almost everywhere, such that $x \in C^{0,1}([0,1])$ holds for any feasible trajectory~\cite[Thm.~1.36]{Weaver2018}. Moreover, it is immediately clear that there is an ellipse $\Omega\subset\mathcal{R}^2$ with focal points $x_O$ and $x_D$, in which any trajectory with minimal flight duration is contained.

Problem~\eqref{eq:continuous-problem} can be numerically solved efficiently with either direct methods using a discretization of the variables to formulate a finite-dimensional nonlinear programming problem~\cite{GeigerEtAl06}, or with indirect methods relying on Pontryagin's maximum principle, leading to a boundary value problem for ordinary differential equations~\cite{Zermelo1931, NgEtAl2011, NgSridharGrabbe2014, JardinBryson2012, MarchidanBakolas2016, Techy2011}. These approaches have also been considered explicitly for free flight planning~\cite{BettsCramer1995, HagelauerMoraCamino1998}.

While the optimal control formulation~\eqref{eq:continuous-problem} is general and convenient for numerical solution of the optimization problem, we will consider a different formulation that is better suited for direct comparison with graph-based approaches here. Assume the flight trajectory $x:[0,T]\to \Omega$ is given by a strictly monotonuously increasing parametrization $t(\tau)$ on $[0,1]$ as $x(t(\tau)) = \xi(\tau)$, and $\xi:[0,1]\to\Omega$ being a Lipschitz continuous path with $\xi(0)=x_O$, $\xi(1)=x_D$. Due to Rademacher's theorem, its derivative $\xt$ exists almost everywhere, and we assume it not to vanish. Then, $t(\tau)$ is defined by the state equation $x_t = v+w \ne 0$ and the airspeed constraint $\|v\| = \ov$, since
\[
	\ov = \|x_t - w \| \quad \text{and}\quad x_t t_\tau = \xt \ne 0
\]
imply
\begin{align}
	& (t_\tau^{-1}\xt -w)^T(t_\tau^{-1}\xt -w) = \ov^2 \notag
	\\
	\Leftrightarrow  \quad & t_\tau^{-2}\xt^T \xt -2 t_\tau^{-1}\xt^T w + w^Tw - \ov^2=0 \notag
	\\
	\Leftrightarrow  \quad & (\ov^2 - w^Tw) t_\tau^2 + 2\xt^T w t_\tau - \xt^T \xt = 0 \notag
	\\
	\Leftrightarrow  \quad & t_\tau = \frac{-\xt^Tw + \sqrt{(\xt^Tw)^2+(\ov^2 - w^Tw)(\xt^T \xt)}}{\ov^2 - w^Tw} =: f(t,\xi,\xt)
	\label{eq:dt-dtau}
\end{align}
due to $t_\tau > 0$. The flight duration $T$ is then given by integrating the ODE~\eqref{eq:dt-dtau} from 0 to 1 as $T=t(1)$. For the ease of presentation let us assume that the wind $w$ is stationary, i.e.~independent of $t$, and thus $f(t,\xi,\xt) = f(\xi,\xt)$. Doing so, we avoid the more complicated work with an ODE. Instead, we obtain
\begin{equation}\label{eq:travel-time}
T(\xi) = \int_0^1 f\big(\xi(\tau),\xt(\tau)\big)\, d\tau.
\end{equation}
We, however, strongly expect our results to directly carry over to the more complex case.
Since the flight duration $T$ as defined in~\eqref{eq:travel-time} is based on a reparametrization $x(t) = \xi(\tau(t))$ of the path such that $\|x_t(t)-w(x(t))\| = \ov$, the actual parametrization of $\xi$ is irrelevant for the value of $T$. Calling two paths $\xi,\hat\xi$ equivalent if there exists a Lipschitz-continuous bijection $r:[0,1]\to[0,1]$ such that $\hat\xi(r(\tau)) = \xi(\tau)$, we can restrict the optimization to equivalence classes $[\xi]$. Thus, the admissible set is
\begin{equation}\label{eq:admissible-set}
	X = \{[\xi] \mid \xi\in C^{0,1}([0,1],\Omega), \; \xi(0)=x_O, \; \xi(1)=x_D\}.
\end{equation}
Since every equivalence class contains a representative with constant ground speed $\|\xt(\tau)\|$,  we will subsequently often assume $\|\xt(\tau)\|=\mathrm{const}$ without loss of generality, such that $\|\xt\|$ is just the length of the flight trajectory. For convenience, let us define the set of representatives with constant ground speed as
\begin{equation*}
	\hat X = \{ \xi \mid [\xi]\in X, \nxt=\mathrm{const} \text{ f.a.a. } \tau\in[0,1] \}.
\end{equation*}
The reduced minimization problem, equivalent to~\eqref{eq:continuous-problem}, now reads
\begin{equation}\label{eq:reduced-problem}
	\min_{[\xi] \in X} T(\xi), \quad\text{or, equivalently,} \quad \min_{\xi\in\hat X} T(\xi).
\end{equation}

\begin{remark}
	Let us interpret this representation of flight duration. In the absence of wind, i.e.~$\|w\|=0$, we obtain $t_\tau = \xt / \ov$. Integrating over $[0,1]$ yields just the total path length divided by the velocity (airspeed and ground speed coincide). For low wind, i.e.~$\|w\|\ll \ov$, we obtain $t_\tau \approx (\xt -\xt^T w) / \ov$, and hence a reduction of flight duration due to the tail wind component $\xt^T w$ (or an increase in the case $\xt^Tw<0$ of head wind). For $\|w\|\to\ov$, we obtain $t_\tau \to \xt / (2 \nxt^{-1} |\xt^T w|)$ in case of a tailwind component $\xt^T w>0$ and $t_\tau\to\infty$ otherwise. In any case, flight duration scales linearly with the length of the path.
\end{remark}

In contrast to the optimal control formulation~\eqref{eq:continuous-problem}, the reduced formulation~\eqref{eq:reduced-problem} allows a direct comparison of continuous and discrete flight trajectories, and is therefore the ideal tool for deriving error bounds in \Cref{sec:error-bounds}. We point out, however, that it is less suited for actually computing an optimal solution.

\subsection{Discrete: airway networks}

If flight trajectories are restricted to certain airways connecting predefined waypoints, flight planning is a special kind of shortest path problem on a graph. Let $V\subset\mathcal{R}^2$ be a finite set of waypoints including $x_O$ and $x_D$, and $E\subset V\times V$ a set of airways such that $G=(V,E)$ is a connected directed graph. A discrete flight path is a finite sequence $(x_i)_{0\le i \le n}$ of waypoints with $(x_{i-1},x_i)\in E$ for $i=1,\dots,n$, connecting $x_0=x_O$ with $x_n=x_D$.

We define a mapping $\Xi:(x_i)_{0\le i\le n} \mapsto [\xi]\in X$ of discrete flight paths to continuous paths by piecewise linear interpolation
\begin{equation}\label{eq:linear-interpolation}
	\xi(\tau) = x_{\lfloor n\tau\rfloor} + (n\tau-\lfloor n\tau\rfloor)(x_{\lceil n\tau\rceil}-x_{\lfloor n\tau\rfloor})
\end{equation}
resulting in polygonal chains, which are Lipschitz-continuous with piecewise constant derivative.
We denote its image $\mathop\mathrm{im} \Xi \subset X$, i.e.\ the set of flight trajectories in the Euclidean plane that can be realized by adhering to the airway network, by $X_G$. The discrete flight planning problem then reads
\begin{equation}\label{eq:discrete-problem}
	\min_{[\xi]\in X_G} T(\xi),
\end{equation}
and differs from its continuous counterpart~\eqref{eq:reduced-problem} only by the admissible set, effectively acting as a particular discretization.

Shortest path problems on static graphs with non-negative weights are
usually solved with the $A^*$ variant of Dijkstra's algorithm \cite{MadkourEtAl2017}. %utilizing a Fibonacci heap. As long as the FIFO property is satisfied, which is the case here, this allows to calculate the exact solution with a time complexity of $\bigO(|A| + |V|\log |V|)$, where $A,V$ are the sets of arcs and vertices, respectively \cite{Dijkstra59, FredmanTarjan87}. \\
%With preprocessing the runtime can be reduced tremendously \cite{BlancoEtAl2016}. However, reducing the number of nodes or arcs by some factor will not affect the runtime-order.

\section{Approximation error bounds}\label{sec:error-bounds}
Having established a setting in which discrete and continuous flight trajectories can be directly compared, we are interested in bounding the suboptimality, i.e.\ the increase of flight duration $T$ relative to the continuous optimum, due to restricting the flight path to predefined airways. In particular, we aim at relating this approximation error to the airway network density.

\subsection{A posteriori error}

For estimating the flight time deviation, we start with a Taylor-based bound in terms of the actual path deviation $\dx = \xi_R-\xi_C$. This bound will serve as the basis for computable bounds in \Cref{sec:computable-error-bounds} and, in addition, provide a quantitative idea of the efficiency of a posteriori error estimators using computable estimates of $\ndx$.

At this point we want to point out that $\xi$, $\dx$, and $\dxt$ are in general functions of $\tau$. In favor of a more compact notation we will usually omit the argument $\tau$ in the remainder of the paper.

\newcounter{lemma-stored}
\setcounter{lemma-stored}{\value{theorem}}
\begin{lemma} \label{lem:second-derivative-bound}
	For any $p\in\Omega$ let $c_0(p)=\|w(p)\|$, $c_1(p)=\|w_x(p)\|$, and $c_2(p)=\|w_{xx}(p)\|$, and assume $c_0 \le \ov / \sqrt{5}$.
	Moreover, let $\xi\in\hat X$, $L:=\xt>0$ and $\underbar{v}^2(p) := \ov^2 - c_0^2(p)$.
	Then the second directional derivative of $f$ as defined in~\eqref{eq:dt-dtau} is bounded by
	\begin{align}
		f''(\xi,\xt) [\dx,\dxt]^2 \le
		\;& \alpha_0(\xi) \ndx^2 + \alpha_1(\xi) \ndx\, \ndxt + \alpha_2(\xi) \ndxt^2
	\end{align}
	for almost all $\tau\in[0,1]$, with $\alpha_i: \Omega\to\mathcal{R}^+$, $i=0,\dots,2$, given as
	\begin{align*}
		\alpha_0(p) &= \frac{L}{\uv^{3}(p)} \left(12 c_1^2(p) + 4 \uv(p) c_2(p)\right)\,, \\
		\alpha_1(p) &= \frac{8 c_1(p)}{\uv^2(p)}\,,\\
		\alpha_2(p) &= \frac{2}{ L \uv(p)}.
	\end{align*}
\end{lemma}
The proof of this lemma is, though not difficult, rather technical and lengthy calculus and is provided in the appendix.

\begin{theorem}\label{th:error-estimate-a-posteriori}
	Let $\xi_C\in \hat X$ be a minimizer of~\eqref{eq:reduced-problem} and $\dx := \xi-\xi_C$. Then there is a constant $r>0$ depending on $\xi_C$ and $w$, such that the \emph{a posteriori} bound
	\begin{align}\label{eq:objective-deviation-integral}
		T(\xi) \le T(\xi_C)
		+ \int_0^1
		\big( & \alpha_0(\xi_C) \ndx^2+ \alpha_1(\xi_C) \ndx \ndxt+ \alpha_2(\xi_C) \ndxt^2
		\big) \, d\tau
	\end{align}
	holds for all paths $\xi\in\hat X$ with $\|\xi-\xi_C\|_{C^{0,1}([0,1])} \le r$ and $\alpha_i$ as defined in \Cref{lem:second-derivative-bound}.
\end{theorem}

\begin{proof}
	We note that $T:C^{0,1}([0,1])^2 \to \mathcal{R}$ as defined in~\eqref{eq:travel-time} is twice continuously Fr{\'e}chet-differentiable at $\xi_C\in \hat X$ due to $\|(\xi_C)_\tau\| = L>0$ for almost all $\tau$. By \Cref{lem:second-derivative-bound}, there are functions $\alpha_0,\alpha_1,\alpha_2$ depending on the local wind $w$ and its derivatives as well as the overall trajectory length $L$, such that
	\begin{multline*}
		f''(\xi_C,(\xi_C)_\tau)[(\dx,\dxt),(\dx,\dxt)] \\
		\le \alpha_0(\xi_C) \ndx^2
		+ \alpha_1(\xi_C) \ndx \ndxt
		+ \alpha_2(\xi_C) \ndxt^2
	\end{multline*}
	holds for almost all $\dx,\dxt\in\mathcal{R}^2$ and $\tau\in[0,1]$. Integrating over $\tau$ yields the bound
	\begin{align*}
		T''(\xi_C)[\dx,\dx]
		\le \int_0^1
		\big(&\alpha_0(\xi_C) \ndx^2  + \alpha_1 (\xi_C)\ndx \ndxt
		+ \alpha_2(\xi_C) \ndxt^2\big) \,
		d\tau
	\end{align*}
	for second directional derivatives of the flight duration $T$ in direction $\dx\in C^{0,1}([0,1])^2$ with $\dx(0)=\dx(1)=0$.
	Due to continuity of $T''$, there exists a neighborhood $B_r(\xi_C)$ of radius $r>0$, such that $T''(\tilde \xi)[\dx,\dx] \le 2\int_0^1
	\alpha_0 \ndx^2
	+ \alpha_1 \ndx \ndxt
	+ \alpha_2 \ndxt^2 d\tau$ for all $\tilde\xi\in B_r(\xi_C)$. Consequently, by Taylor's theorem and using $\dx=\xi-\xi_C$, we can bound
	\begin{align*}
		T(\xi)
		&= T(\xi_C) + \underbrace{T'(\xi_C)\dx}_{=0} + \int_0^1 (1-\nu) T''(\xi_C + \nu\dx)[\dx,\dx] \, d\nu \\
		&\le T(\xi_C) + \int_0^1
		\big(
		\alpha_0(\xi_C) \ndx^2
		+ \alpha_1(\xi_C) \ndx \ndxt
		+ \alpha_2(\xi_C) \ndxt^2\big) \,d\tau
	\end{align*}
	due to $\xi_C$ being a minimizer.
\end{proof}

\subsection{Trajectory approximation in locally dense graphs}
The approximation error of the optimal discrete flight path $\xi_G$ according to~\eqref{eq:discrete-problem} relative to the continuous optimum $\xi_C$ of~\eqref{eq:reduced-problem} due to the smaller admissible set $X_G \subset X$ depends on the density of the airway network. The discussion will be limited to a certain class of locally dense digraphs as defined in~\cite{BorndoerferDaneckerWeiser2021}.

\begin{definition}
	A digraph $G=(V,E)$ is said to be $(h,l)$-dense in a convex set $\Omega\subset\mathcal{R}^2$ for $h,l\ge 0$, if it satisfies the following conditions:
	\begin{enumerate}
		\item \emph{containment:} $V\subset\Omega$
		\item \emph{vertex density:}  $\forall p\in\Omega:  \exists v\in V: \|p-v\| \le h$
		\item \emph{local connectivity:}  $\forall v,w\in V, \|v-w\|\le l+2h: (v,w)\in E$
	\end{enumerate}
\end{definition}
An example for such an airway digraph is shown in \Cref{fig:graph-connectedness}. Note that, even for $l\rightarrow 0$, the minimum local connectivity length of $2h$ guarantees that a vertex is connected to its neighbors. It is easy to show that any $(h,l)$-dense digraph is connected, such that a path from origin to destination exists.

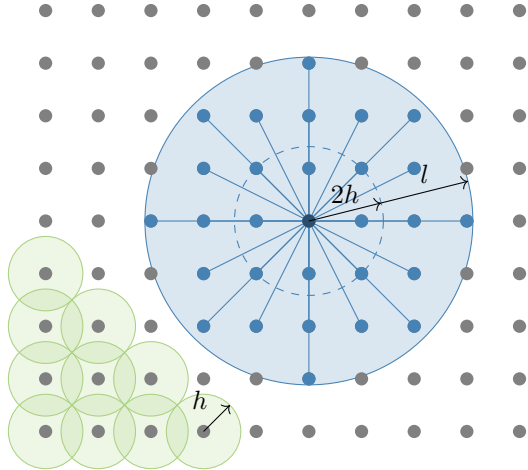
\begin{figure}[b]
	\centering
	\begin{tikzpicture}[
			scale=0.7,
			griddot/.style={
				fill=gray, circle,
				minimum size=5, inner sep=0pt, outer sep=0pt},
			centerdot/.style={
				fill=darkblue,circle,
				minimum size=5, inner sep=0pt, outer sep=0pt},
			neighbordot/.style={
				fill=steelblue, circle,
				minimum size=5, inner sep=0pt, outer sep=0pt},
			hdisc/.style={
				lightgreen, fill=lightgreen, fill opacity=0.2},
			ldisc/.style={
				steelblue, fill=steelblue, fill opacity=0.2}
		]

		\draw[steelblue, dashed] (3,5) circle (1.414);
		\draw[ldisc] (3,5) circle (1.7+1.414);

		% Green discs
		\draw[hdisc] (-2,1) circle (0.707);
		\draw[hdisc] (-1,1) circle (0.707);
		\draw[hdisc] ( 0,1) circle (0.707);
		\draw[hdisc] ( 1,1) circle (0.707);
		\draw[hdisc] (-2,2) circle (0.707);
		\draw[hdisc] (-1,2) circle (0.707);
		\draw[hdisc] ( 0,2) circle (0.707);
		\draw[hdisc] (-2,3) circle (0.707);
		\draw[hdisc] (-1,3) circle (0.707);
		\draw[hdisc] (-2,4) circle (0.707);

		% All nodes
		\foreach \x in {-2,-1,...,7}{
			\foreach \y in {1,2,...,9}{
				\node[griddot] at (\x,\y){};
			}
		}

		% Neighbor nodes
		\foreach \x in {1,2,...,5}{
			\foreach \y in {3,4,...,7}{
				\node[neighbordot] at (\x,\y){};
				\draw[-, steelblue] (\x,\y) -- (3,5);
			}
		}

		\node[neighbordot] at (6,5){};
		\node[neighbordot] at (3,2){};
		\node[neighbordot] at (0,5){};
		\node[neighbordot] at (3,8){};

		\draw[-, steelblue] (6,5) -- (3,5);
		\draw[-, steelblue] (3,2) -- (3,5);
		\draw[-, steelblue] (0,5) -- (3,5);
		\draw[-, steelblue] (3,8) -- (3,5);

		\node[centerdot] at (3,5){};

		% Annotations
		\draw[->] (1,1) -- node[above left] {$h$}  (1+0.5,1+0.5);

		\draw[->] (3,5) -- node[above] {$2h$}  (4.372, 5.343);
		\draw[->] (4.372, 5.343) -- node[above] {$l$} (6.021, 5.755);

	\end{tikzpicture}
	\caption{A locally densely connected digraph with cartesian structure. The center node (dark blue) is connected to all nodes in a circular neighborhood of radius $2h+l$ (light blue) with edges in both directions.} \label{fig:graph-connectedness}
\end{figure}

Let $\xi_C\in X$ be a global minimizer of the continuous problem formulation~\eqref{eq:reduced-problem}, and $\xi_G\in X_G$ be a shortest discrete path in the $(h,l)$-dense airway digraph $G$ satisfying~\eqref{eq:discrete-problem}.
For establishing a bound for the excess flight duration in terms of the airway density, we first construct a particular discrete path $\xi_R(\xi_C)\in X_G$ using a rounding procedure, and derive a bound for $T(\xi_R)-T(\xi_C)\le e(h,l)$, from which the actual error bound $T(\xi_G)-T(\xi_C)\le e(h,l)$ immediately follows from optimality of $\xi_G$.

For defining $\xi_R(\xi)\in X_G$, with an $(h,l)$-dense digraph $G$ with $l>0$, for a given continuous path $\xi\in \hat X$ with $\xt=\mathrm{const}$, we first choose an equidistant grid $\tau_i = i/n$, $n=\lceil \xt / l \rceil$, for $i=0,\dots,n$. By construction, the distance of the corresponding trajectory points is bounded by $\|\xi(\tau_i)-\xi(\tau_{i+1})\| \le l$.
For each $i$, there is some $v_i \in V$ with $\|v_i - \xi(\tau_i)\| \le h$, such that $\|v_i - v_{i+1}\| \le l+2h$. Consequently, $(v_i,v_{i+1})\in E$, and $(v_i)_{0\le i \le n}$ is a valid discrete path, for which we define $[\xi_R] = \Xi (v_i)_i$.

It is intuitively clear -- and rigorously confirmed below -- that the excess flight duration $T(\xi_R)-T(\xi_C)$ is affected by both, the spatial distance between $\xi_R$ and $\xi_C$, e.g., taking a longer detour or flying through an area with adverse wind conditions, and the angular deviation, e.g., a zigzag path tends to take longer than a straight trajectory. In order to capture these effects, we will first bound the spatial distance $\|\xi_R-\xi_C\|_{L^\infty[0,1]}$ and the angular deviation $\|(\xi_R-\xi_C)_\tau\|_{L^\infty[0,1]}$, and equip the space of Lipschitz-continuous functions with the norm $\|f\|_{C^{0,1}([0,1])} := \|f\|_{L^\infty([0,1])} + \|f_\tau\|_{L^\infty([0,1])}$.

\begin{theorem} \label{th:path-error}
	Assume $\xi\in \hat X \cap C^{1,1}([0,1],\Omega)$ has bounded curvature, i.e. there is some $\bar{\sigma}$ with $\|\xt(a)-\xt(b)\| \le \os |a-b|$ for $a,b\in[0,1]$, and denote the length of the trajectory by $L=\xt$. Then, the following bounds hold for the discrete approximation $\xi_R(\xi)$ in an $(h,l)$-dense digraph:
	\begin{align}
		&\text{(distance error)} &\|\xi_R(\xi)-\xi\|_{L^\infty([0,1])} &\le \frac{\os l^2}{8L^2} + h\,, \label{eq:distance-error} \\
		&\text{(angular error)} & \|(\xi_R(\xi)-\xi)_\tau\|_{L^\infty([0,1])} &\le \frac{\sqrt{2}\os l}{L} + 2h\left(\frac{L}{l} +1\right). \label{eq:angular-error}
	\end{align}
	If $l\le L$, we obtain the total error bound
	\begin{align}\label{eq:path-error-bound}
		\|\xi_R(\xi)-\xi\|_{C^{0,1}([0,1])}
		&\le \left(\frac18+\sqrt{2}\right) \bar{\sigma}\frac{l}{L} + 2h\frac{L}{l} + 3h \nonumber\\
		&\le 2 \bar{\sigma}\frac{l}{L} + 2h\frac{L}{l} + 3h.
	\end{align}
\end{theorem}

\begin{proof}
	Let $\hat \xi(\tau) = \xi(\tau_{\lfloor n\tau\rfloor}) + (n\tau-\lfloor n\tau\rfloor)(\xi(\tau_{\lceil n\tau\rceil})-\xi(\tau_{\lfloor n\tau\rfloor}))$ be the linear interpolant of the continuous trajectory $\xi$ on $n=\lceil L / l\rceil$ equisized intervals. Standard interpolation error estimates yield
	\[
		\|(\hat\xi - \xi)(\tau)\| \le \os / (8n^2) \le \frac{\os l^2}{8L^2}
	\]
	for all $\tau$~\cite[Ch.\ 3.1, p.\ 93 ff.]{AtkinsonHan2009}.
	Moreover, with $\alpha = (n\tau-\lfloor n\tau\rfloor) \in [0,1]$,
	\begin{align}
		\hat\xi(\tau)-\xi_R(\tau)
		&= \xi(\tau_{\lfloor n\tau\rfloor}) - x_{\lfloor n\tau\rfloor}
		+ \alpha \Big( \xi(\tau_{\lceil n\tau\rceil})-x_{\lceil n\tau\rceil}-\xi(\tau_{\lfloor n\tau\rfloor})+x_{\lfloor n\tau\rfloor}  \Big) \notag \\
		&= (1-\alpha)\big( \xi(\tau_{\lfloor n\tau\rfloor})- x_{\lfloor n\tau\rfloor} \big)
		+ \alpha \big( \xi(\tau_{\lceil n\tau\rceil})-x_{\lceil n\tau\rceil} \big) \label{eq:offset-error}
	\end{align}
	implies $\|(\hat \xi - \xi_R)(\tau)\| \le h$, which yields the distance error bound~\eqref{eq:distance-error} by triangle inequality.

	Let $\phi = (\hat\xi - \xi)_k$, $k\in\{1,2\}$, be one of the two components of the difference between continuous trajectory and linear interpolant. By the mean value theorem, there is a point $\hat\tau\in\mathopen]\tau_i,\tau_{i+1}\mathclose[$ with $\phi_\tau(\hat\tau) = 0$. Thus,
	\[
		|\phi_\tau(\tau)|
		=  |\phi_\tau(\tau) -\phi_\tau(\hat\tau)|
		\le \frac{\os}{n} \quad \forall \tau\in[\tau_i,\tau_{i+1}]
	\]
	holds for all~$i=0,\dots,n-1$ and implies $\|(\hat\xi-\xi)_\tau(\tau)\| \le \sqrt{2}\os / n \le \sqrt{2}\os l / L$ for all~$\tau$.
	Moreover,~\eqref{eq:offset-error} implies
	\[
		(\hat\xi-\xi_R)_\tau(\tau) = -n \big( \xi(\tau_{\lfloor n\tau\rfloor})- x_{\lfloor n\tau\rfloor} \big)
		+ n \big( \xi(\tau_{\lceil n\tau\rceil})-x_{\lceil n\tau\rceil} \big)
	\]
	and therefore $\|(\hat\xi-\xi_R)_\tau\| \le 2nh \le 2h(L/l +1)$ and yields the angular error bound~\eqref{eq:angular-error} by triangle inequality.
\end{proof}

Of course, if $l < l'$, then the $(h,l)$-dense digraph $G$ is a subgraph of the $(h,l')$-dense digraph $G'$, provided their vertex sets coincide. Thus, the discretization error of a shortest path in $G'$ is less or equal to one in $G$ -- a fact that is not reflected by \Cref{th:path-error}. The reason is the explicit rounding procedure, which tends to select arcs of length $l'$ in $G'$ even if shorter arcs of length $l$ would be better. This effect can be essentially avoided if the connectivity length $l$ is chosen sufficiently small compared to the path length. It should not be chosen too small compared to $h$, however, because then the angular error can dominate, as the following pathological example shows.

\begin{figure}
	\centering
	\begin{tikzpicture}[
			scale=3,
			griddot/.style={
				fill=gray, circle,
				minimum size=3, inner sep=0pt, outer sep=0pt},
		]
		\draw[->, line width=2] (-0.2,0) -- (1.2,0);
		\draw[->, line width=2] (-0.2,-1.1) -- (-0.2,1.1);

		% y-ticks
		\foreach \y in {-1,0,1}
		{
			\draw[line width=0.5] ($(-0.2,\y)+(0.04,0)$) -- ($(-0.2,\y)-(0.04,0)$);
		}
		\node () at (-0.3,-1) {-1};
		\node () at (-0.3, 0) {0};
		\node () at (-0.3, 1) {1};

		% Continuous trajectory
		\draw[lightgreen, line width=3] (0,0) -- (1,0);

		% x-ticks
		\foreach \x in {0,,1}
		{
			\draw[line width=0.5] ($(\x,0)+(0,0.05)$) -- ($(\x,0)-(0,0.05)$);
		}
		\node () at (0, -0.11) {0};
		\node () at (1, -0.11) {1};

		% Minor x-ticks
		\foreach \x in {0,0.05,...,1}
		{
			\draw[line width=0.5] ($(\x,0)+(0,0.015)$) -- ($(\x,0)-(0,0.015)$);
		}

		% The nodes
		\foreach \x in {-0.15,-0.05,...,1.1}{
			\foreach \y in {1}{
				\node[griddot] at (\x,\y){};
			}
		}
		\foreach \x in {-0.2,-0.1,0,0.1,0.2,...,1.1}{
			\foreach \y in {-1}{
				\node[griddot] at (\x,\y){};
			}
		}

		\node[griddot, minimum size=6] at (0,0){};
		\node[griddot, minimum size=6] at (1,0){};

		\draw[decorate,decoration={brace,amplitude=2pt}] (0.05,1.03) -- (0.15,1.03);
		\node () at (0.1,1.12) {$2l$};

		% The rounded path
		\draw[gray, line width=1]
		(0.00, 0)
		-- (0.05, 1)
		-- (0.10,-1)
		-- (0.15, 1)
		-- (0.20,-1)
		-- (0.25, 1)
		-- (0.30,-1)
		-- (0.35, 1)
		-- (0.40,-1)
		-- (0.45, 1)
		-- (0.50,-1)
		-- (0.55, 1)
		-- (0.60,-1)
		-- (0.65, 1)
		-- (0.70,-1)
		-- (0.75, 1)
		-- (0.80,-1)
		-- (0.85, 1)
		-- (0.90,-1)
		-- (0.95, 1)
		-- (1.00, 0);
	\end{tikzpicture}
	\caption{Illustration of \Cref{ex:pathological}. Green: continuous trajectory $\xi$, gray: rounded path $\xi_R$.} \label{fig:pathological_example}
\end{figure}
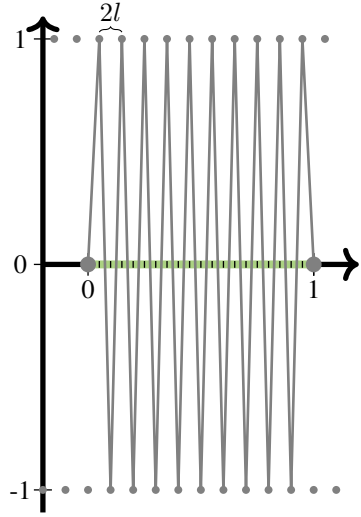

\begin{example}\label{ex:pathological}
	Consider $\xi(\tau) = [\tau,0]^T$ and
	\[
		V = \{[l(2i+j),2j-1]^T \mid i,j\in\Z\} \cup \{[0,0]^T, [1,0]^T\}
	\]
	with $l\ll 1$ and $h=\sqrt{1+l^2/4}\approx 1$. Rounding to the nearest vertex yields a discrete zigzag path with length at least $2h/l$, as illustrated in \Cref{fig:pathological_example}. Thus, the bounds~\eqref{eq:angular-error} and~\eqref{eq:path-error-bound} are asymptotically sharp for $l\to 0$.
\end{example}

Hence, we select a theoretically optimal $l$ by minimizing the error bound~\eqref{eq:path-error-bound}.

\begin{theorem}\label{th:path-error-condensed}
	Under the assumptions of \Cref{th:path-error}, including $l\le L$, the choice
	\[
		l = L \sqrt{\frac{h}{\bar{\sigma}}}
		\quad \Leftrightarrow \quad
		h = \bar{\sigma} \frac{l^2}{L^2}
	\]
	is optimal with respect to the error bound~\eqref{eq:path-error-bound} and yields the bounds
	\begin{align}
		\|\xi_R(\xi)-\xi\|_{L^\infty([0,1])} &\le \frac{9\os l^2}{8L^2} \label{eq:path-deviation-0} \\
		\|(\xi_R(\xi)-\xi)_\tau\|_{L^\infty([0,1])} &\le \frac{11\os l}{2L} \label{eq:path-deviation-1}\\
		\|\xi_R(\xi)-\xi\|_{C^{0,1}([0,1])} &\le 7\bar{\sigma} \frac{l}{L}.
	\end{align}
\end{theorem}

\begin{proof}
	Straightforward minimization of~\eqref{eq:path-error-bound} yields the given optimal choice of $l$. Inserting this into~\eqref{eq:distance-error}, \eqref{eq:angular-error}, and the bound~\eqref{eq:path-error-bound} and using $l\le L$ yields the claims.
\end{proof}

The pathological \Cref{ex:pathological} reveals a further limitation of the derivation of bounds by employing an explicit rounding procedure: the length of the rounded path $\xi_R$ can be much larger than the length of the discretely optimal path $\xi_G$. In the example this is $\mathcal{O}(2/l)\to \infty$ compared to $\mathcal{O}(1)$, with $\xi_G$ connecting the vertices along the horizontal line $[0,1]\times\{1\}$. We point out that this susceptibility of the bound to pathological worst cases is structurally similar to common a priori error estimates for finite element methods~\cite{DeuflhardWeiser2012}. Nevertheless, even if the angular error responsible for the pathological behavior is ignored, the same optimal order of $h=\mathcal{O}(l^2)$ is obtained.

\subsection{Computable error bounds} \label{sec:computable-error-bounds}

\begin{theorem}\label{cor:error-bound-local}
	Assume that $\xi_C\in \hat X \cap C^{1,1}([0,1])^2$ is a minimizer of~\eqref{eq:reduced-problem} with bounded curvature, i.e. there is $\os<\infty$ such that $\|\xt(a)-\xt(b)\| \le \os |a-b|$ for all $a,b\in[0,1]$. Let $L=\|(\xi_C)_\tau\|$ denote the length of the optimal flight trajectory and $\oa_i := \max_{\tau\in[0,1]} \alpha_i(\xi_C(\tau))$ with $\alpha_i$ defined in \Cref{lem:second-derivative-bound}. Then, there is a constant $r>0$, such that the $\emph{local}$ bound
	\begin{align}
		T(\xi_G)-T(\xi_C)
		\le \frac{4\os^2 l^2}{3L^2} \left( \frac{l^2}{L^2} \oa_0 + 5\frac{l}{L}\oa_1
		+ 23 \oa_2  \right)
		\le \frac{92\os^2 \oa_2}{3L^2} l^2 +\bigO(l^3)
		\label{eq:error-bound-local}
	\end{align}
	holds for all $(h,l)$-dense digraphs with $l\le \min\big\{\frac{r}{7\bar{\sigma}},1\big\}L$ and $h\le\frac{\bar{\sigma}l^2}{L^2}$.
\end{theorem}
\begin{proof}
	Inserting the bounds~\eqref{eq:path-deviation-0} and~\eqref{eq:path-deviation-1} from \Cref{th:path-error-condensed} into the claim~\eqref{eq:objective-deviation-integral}, we obtain
	\begin{align*}
		T(\xi_R) - T(\xi_C)
		& \le \int_0^1
		\bigg(  \alpha_0 \frac{81\os^2l^4}{64L^4}
		+\, \alpha_1 \frac{9\os l^2}{8L^2} \frac{11\os l}{2L}
		+\, \alpha_2 \frac{121\os^2 l^2}{4L^2}
		\bigg) \, d\tau \\
		&< \frac{4\bar \sigma^2 l^2}{3L^2} \int_0^1
		\bigg(  \alpha_0 \frac{l^2}{L^2}
		+\, 5\alpha_1 \frac{l}{L}
		+\, 23 \alpha_2
		\bigg) \, d\tau
	\end{align*}
	since $l \le  \min\big\{\frac{r}{7\bar{\sigma}},1\big\}L$, where $r$ is the neighborhood radius from \Cref{th:error-estimate-a-posteriori} and $\alpha_i$ provided by \Cref{lem:second-derivative-bound}. Inserting the upper bounds $\oa_i$ for $\alpha_i$ yields the claim.
\end{proof}

Note that the bound holds in a certain neighborhood of a continuous minimizer $\xi_C$ and therefore bounds the asymptotic error behavior for $h,l\to 0$, rather than providing a globally reliable error bound.

We can go one step further and eliminate the dependence on the actual optimal path $\xi_C$ by choosing appropriate global bounds on the constants and route properties. For that, we define the global bounds
\[
	\oc_0 := \|w\|_{L^\infty(\Omega)}, \quad \oc_1 := \|w_x\|_{L^\infty(\Omega)}, \quad \text{and}\quad
	\oc_2 := \|w_{xx}\|_{L^\infty(\Omega)}
\]
for the wind and its derivatives.

\newcounter{lemma-stored-2}
\setcounter{lemma-stored-2}{\value{theorem}}
\begin{lemma}\label{lem:curvature-bound}
	Let $\xi_C \in\hat X$ be a minimizer of~\eqref{eq:reduced-problem}. Then, it is twice continuously differentiable and its second derivative is bounded by
	\begin{equation}
		\|(\xi_C)_{\tau\tau}\| \le \os :=
		\frac{\oc_1L^2}{\ov - \oc_0 }\left( \sqrt{2}\,\ov    + \frac{\ov
			+ \oc_0}{\ov - \oc_0}     \left((1+\sqrt{2})\ov  + \oc_0\right)   \right).
	\end{equation}
	For $\oc_0 \le \ov / \sqrt{5}$ this simplifies to $ \os \le 17\oc_1L^2$.
\end{lemma}
Again, the proof of this Lemma is rather long and can be found in the appendix.

\begin{lemma}\label{lem:pathlength-bound}
	Assume that $\xi_C$ is a global minimizer of~\eqref{eq:reduced-problem} with path length $L$ and that $\oc_0  \le \ov / \sqrt{5}$. Then
	\begin{align}
		\|x_D-x_O\| \le L \le \frac{\ov + \oc_0}{\ov - \oc_0}\|x_D-x_O\| < \frac{8}{3}\|x_D-x_O\|.
	\end{align}
\end{lemma}

\begin{proof}
	The lower bound is clear, since the trajectory can not be shorter than the straight connection. The flight time $T_s$ on the straight line is at most $\frac{\|x_D-x_O\|}{\ov-\oc_0}$. Since $\xi_G$ is optimal, we obtain
	\[
		T_s \ge T(\xi_C) \ge \frac{L}{\ov+\oc_0},
	\]
	which yields the upper bound for $L$.
\end{proof}

We can now completely eliminate the need for a posteriori information about $\xi_C$ and derive an \emph{a priori} error bound.
\begin{theorem}\label{cor:error-bound-apriori}
	Let $\xi_C\in \hat X \cap C^{1,1}([0,1])^2$ be a global continuous minimizer and $\xi_G$ be a shortest path in the $(h,l)$-dense graph $G$. Moreover, let  $\tilde L := \|x_D - x_O\|$ and assume $\oc_0 \le \ov / \sqrt{5}$. Then, with $\os$ from \Cref{lem:curvature-bound},
	\begin{align}\label{eq:error-bound-global}
		T(\xi_G)-T(\xi_C)
		& \le \frac{4\os^2}{3\tilde L^3 \ov} \left(
						14 l^2 \left(\frac72\oc_1^2 + \ov\oc_2\right)
						+ \frac{51\oc_1l}{\ov}
						+ 52 \right) \; l^2
%		\\
%		& \le \frac{142488\oc_1 }{\tilde L \ov} l^2 + \mathcal{O}(l^3)
		\\
		& \le 1.5\cdot 10^5 \frac{\oc_1 }{\tilde L \ov} l^2 + \mathcal{O}(l^3)
	\end{align}
	holds for sufficiently dense graphs.
\end{theorem}
\begin{proof}
	For $\uv(p) = \sqrt{\ov^2 - c_0^2(p)}$ we obtain $8\ov /9 < \underline v \le \ov$.
	\Cref{lem:pathlength-bound} together with $\alpha_i$ from \Cref{lem:second-derivative-bound} now yields the global bounds
	\begin{align*}
		\alpha_0(p)
		&\le \frac{8\tilde L}{3\uv(p)^3} (12\oc_1^2 + 4\uv(p)\oc_2)
		\le \frac{14\tilde L}{\ov} \left(\frac72\oc_1^2 + \oc_2\ov\right)
		=: \tilde\alpha_0,
		\\
		\alpha_1(p)
		&\le \frac{8\oc_1}{\uv(p)^2} \le \frac{81\oc_1}{8\ov^2}
		=: \tilde\alpha_1, \quad\text{and}
		\\
		\alpha_2(p)
		&\le \frac{2}{\tilde L \uv(p)} \le \frac{9}{4\tilde L \ov}
		=: \tilde\alpha_2.
	\end{align*}
	Inserting them into~\eqref{eq:error-bound-local} provides the bound
	\begin{align*}
		T(\xi_G) - T(\xi_C)
		&\le T(\xi_R) - T(\xi_C) \\
		&\le \frac{4\os^2 l^2}{3\tilde L^2} \left(\frac{l^2}{\tilde L^2}\tilde\alpha_0
		+ 5\frac{l}{\tilde L} \tilde\alpha_1 + 23\tilde\alpha_2\right) \\
		&\le \frac{4\os^2 l^2}{3\tilde L^2} \left(
			\frac{14l^2}{\tilde L \ov} \left(\frac72\oc_1^2 + \oc_2\ov\right)
			+ \frac{405\oc_1l}{8\ov^2\tilde L}
			+ \frac{207}{4\tilde L \ov} \right)
			\\
		&\le \frac{4\os^2 l^2}{3\tilde L^3 \ov} \left(
			14l^2 \left(\frac72\oc_1^2 + \oc_2\ov\right)
			+ \frac{51\oc_1l}{\ov}
			+ 52 \right),
	\end{align*}
	which completes the proof.
\end{proof}

\section{Numerical Examples} \label{sec:numeric}
In \Cref{sec:error-bounds} we derived three error bounds: i)~the computationally in general unavailable \emph{a posteriori} bound~\eqref{eq:objective-deviation-integral}, ii)~the \emph{local} bound~\eqref{eq:error-bound-local}, and iii)~the \emph{a priori} bound~\eqref{eq:error-bound-global}. Now we validate these bounds with the four test problems from~\cite{BorndoerferDaneckerWeiser2021}. For this comparison and for evaluating the \emph{a posteriori} bound we compute the optimal continuous trajectory $\xi_G$ numerically using a direct collocation approach to high accuracy.

\subsection{Test instances}
The goal in all four test instances is to find a time-optimal trajectory from $x_O=[0,0]^T$ to $x_D=[1,0]^T$ through wind fields of varying spatial frequency, see \Cref{fig:test-cases}. The wind speed is always bounded by $\|w\|_{L^\infty(\mathcal{R}^2)} \le \bar w = 0.5 \ov$. All values are chosen dimensionless, i.e. $\ov=1$. For the first test problem a) we define the laminar shear flow
\begin{equation*}
	w(p) =
	\begin{bmatrix}
	\bar w\min(\max( 2\frac{p_2}{H}{-}1,{-}1),1) \\
	0
	\end{bmatrix}\,,
\end{equation*}
with $H=0.5$, see \Cref{fig:test-cases}~a). In problems b)-d), the wind $w$ is the sum of an increasing number of non-overlapping vortices $w_i$, each of which is described by
\begin{equation*}
	w_i(p) = s_i \tilde w_i(r_i)
	\begin{bmatrix}
	- \sin(\alpha_i) \\
	\cos(\alpha_i)
	\end{bmatrix}\,,
\end{equation*}
where $s_i$ is the spin of the vortex ($s_i{=}{+}1$: counter-clockwise, $s_i{=}{-}1$: clockwise), $r_i = \|p-z_i\|_2$ is the distance from the vortex center $z_i$, $\alpha_i$ is the angle between $p$, $z_i$, and the positive $x$-axis with $\tan(\alpha_i) = \frac{(p-z_i)_2}{(p-z_i)_1}$ and the absolute vortex wind speed $\tilde w_i$ is a function of $r$ and the vortex radius $R_i$:
\begin{align*} \label{eq:wind_vortex}
	\tilde w_i(r) =
	\left[ \begin{array}{cl}
	\bar w \exp\left(\frac{(r/R_i)^2}{(r/R_i)^2 - 1}\right) & \text{if}~r<R_i\\
	0 & \text{otherwise}
	\end{array}\right]\,.
\end{align*}
More precisely, problems b)-d) involve 1, 15, and 50 regularly aligned vortices with $R{=}1/2$, 1/8, and~1/16, respectively, see \Cref{fig:test-cases}~b)-d). Vortices with positive spin (counter-clockwise) are colored green, vortices with negative spin (clockwise) are colored red.

Note that at least problem~d) is clearly an exaggeration, as no commercial plane would ever try to traverse a wind field like this. We use this instance to provide evidence that our claims hold even under the most adverse conditions.

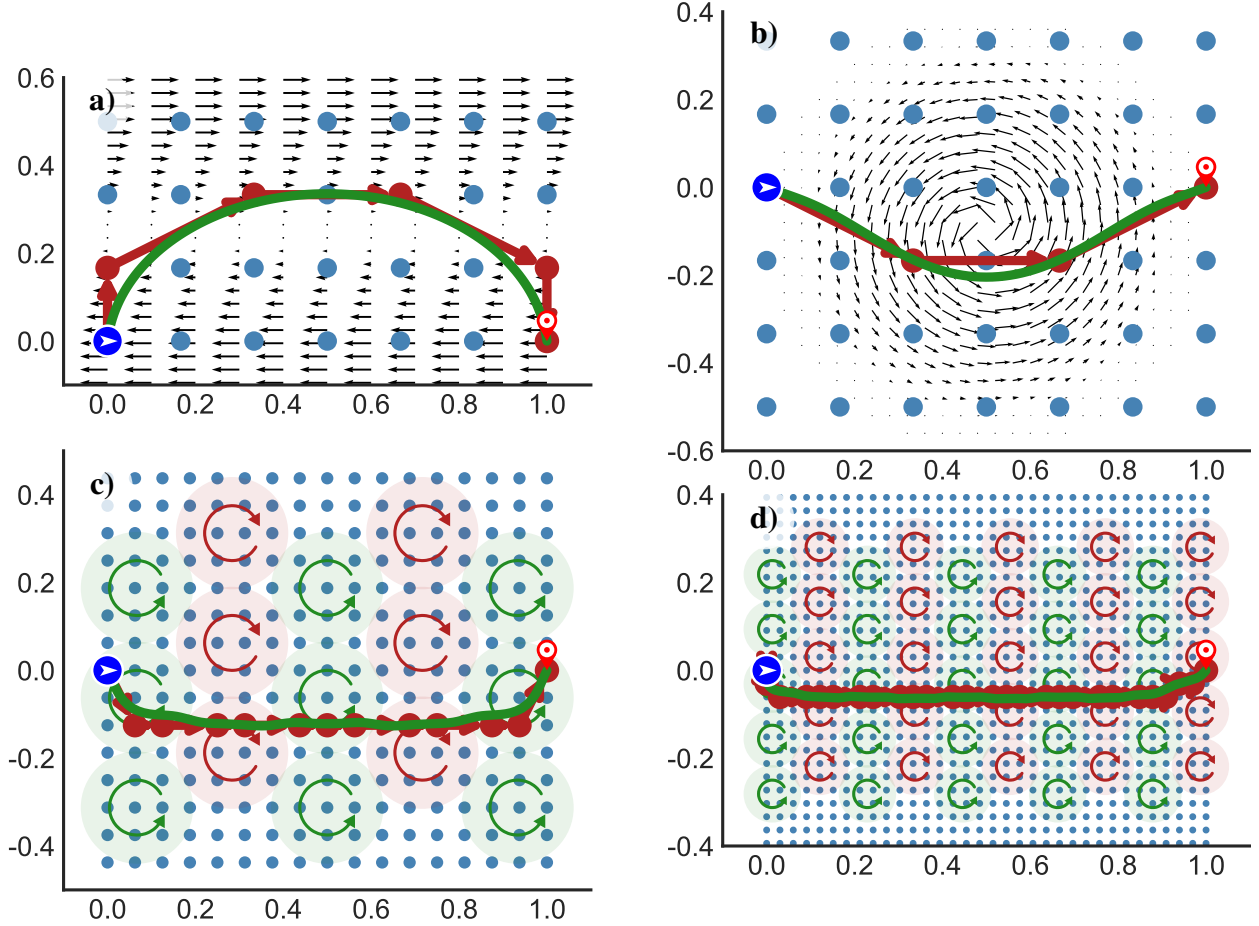
\begin{figure}[ht]
	\begin{tikzpicture}[
			item/.style={
				circle,
				minimum size=0.8cm,
				inner sep=0pt,
				fill=white,
				fill opacity=0.8,
				draw opacity=1,
				text opacity=1,
				font=\bfseries\large}
		]

		\node[inner sep=0pt] at (0,0)
		{\includegraphics[width=\textwidth]{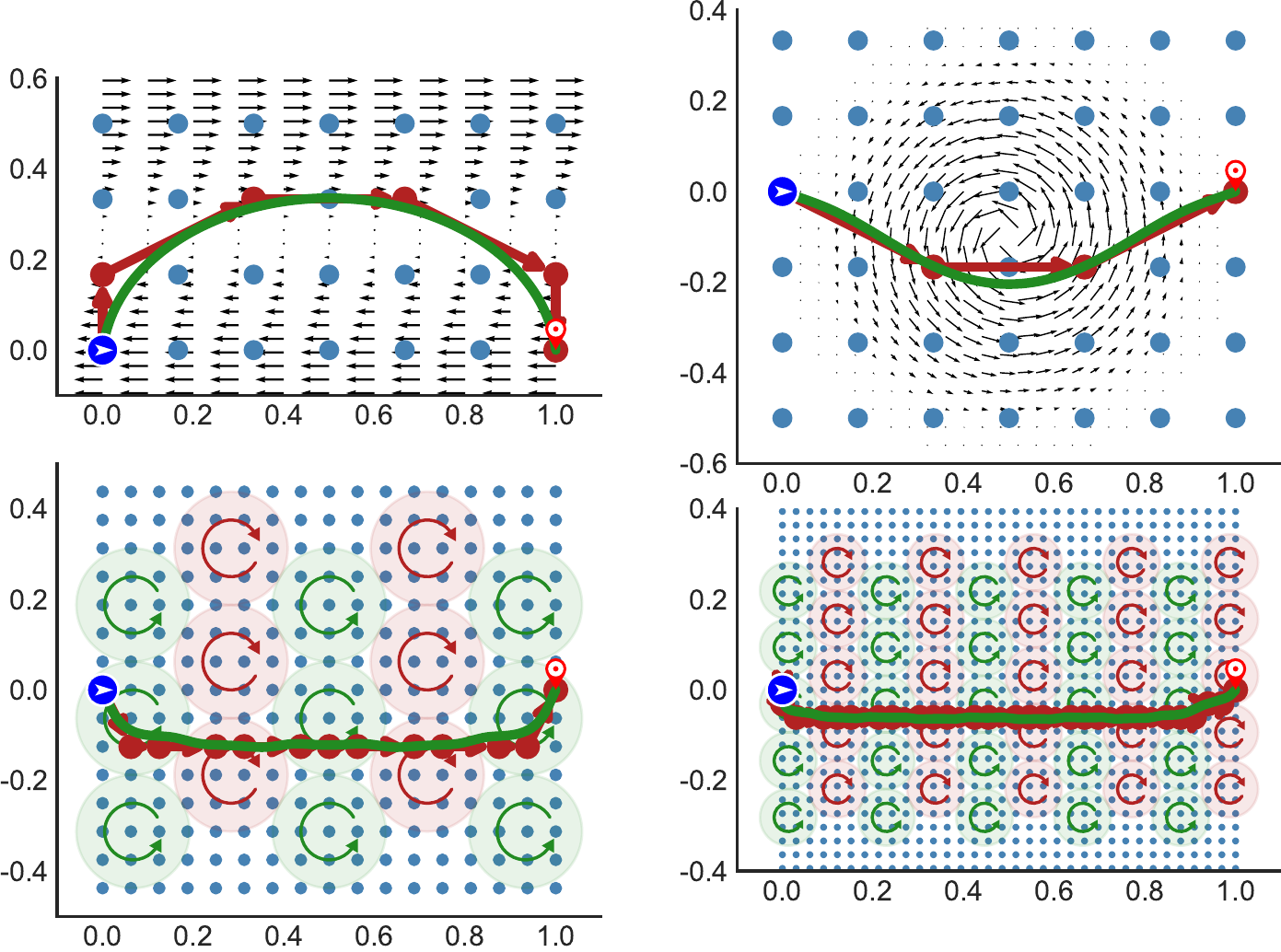}};

		\node[item] at (-7.0, 4.8) {a)};
		\node[item] at ( 1.8, 5.7) {b)};
		\node[item] at (-7.0,-0.3) {c)};
		\node[item] at ( 1.8,-0.75) {d)};
	\end{tikzpicture}
	\caption{
		Test problems a)-d), $x_O=[0,0]^T$, $x_D=[1,0]^T$, $\bar w=\frac12 \ov$. Blue dots: $(h,l)$-dense rectangular network-graph of some exemplary density and connectivity, satisfying $h=\os l^2/L^2$, red: shortest path on the graph, green: continuous optimal trajectory. Note that in every case the straight trajectory is particularly unfavorable.}
	\label{fig:test-cases}
\end{figure}

\subsection{Results}
The three error bounds (i)~\emph{a posteriori}, ii)~\emph{local}, and iii)~\emph{a priori}) involve a decreasing amount of information about the optimal trajectory $\xi_C$. Hence it is not surprising that the latter is far from being sharp and overestimates the actual error by several orders of magnitude. This is mainly due to the worst case estimates that must be considered in several steps finding globally valid constants. Most importantly, the wind speed and its derivatives need to be bounded globally, even though the conditions experienced along the actually flown path are usually much easier, as can be seen in \Cref{fig:test-cases}. Thus, a more interesting question is,
how much can be gained by incorporating a posteriori information?

In order to answer this, we evaluate the error between a graph-based shortest path and the continuous optimum for various graph densities, respecting the optimal combination of node density and connectivity $h=\os l^2/L^2$ as stated in \Cref{th:path-error-condensed}. The results are shown in the double-logarithmic plots of \Cref{fig:error-bound-three-ways} as blue dots.

The three bounds are illustrated in i)~gray, ii)~purple and iii)~red. In the first case, the a posteriori bound~\eqref{eq:objective-deviation-integral}, we depict the three individual parts of the integral by colored areas, from top to bottom: $\int\alpha_0 \ndx^2 d\tau$, $\int\alpha_1 \ndx \ndxt d\tau$, $\int\alpha_2 \ndxt^2 d\tau$.
\Cref{fig:integral-error-bound-shares} reveals that the relative distribution of these parts is more or less constant over a wide range of graph densities. This suggests that the theoretically optimal choice of $h,l$ balances the error terms against each other evenly.

It needs to be mentioned, that, because we have only $\|w\|/\ov \le 0.5$ here (not $\le 1/\sqrt 5$), we cannot use $\alpha_i$ as stated in \Cref{lem:second-derivative-bound}, but must revert to the results from \Cref{th:hessian-bound} in the appendix. Since the purpose of that Lemma is solely to provide a more compact notation, however, this should not be of any concern. For the same reason, the coefficients in~\eqref{eq:error-bound-global} also need to be adjusted accordingly.

We show linear trend lines for bound i) and the experimental results,  excluding the data of the 10\% sparsest graphs. Results in that region are dominated by effects of local minima (e.g. the continuous optimum goes left, while the discrete path goes right, which results in a large distance error). As our error bounds were developed to hold in a certain neighborhood of $\xi_C$, we ignore these effects. It is, however, interesting to see that the bounds hold anyway. For the same reason we exclude these data from \Cref{fig:integral-error-bound-shares}.

As a first important result we point out that the quadratic order of the derived error bounds ii) and iii), $T(\xi_C)-T(\xi_G)\in\bigO(l^2)$, matches the numerical results satisfyingly well. The fitted exponents are listed in \Cref{tab:error-bound-trends}.

Further, starting from the \emph{a priori} bound iii), we note that the bound can be tightened significantly by incorporating a posteriori knowledge. With the \emph{local} approach, the bound can already be improved by roughly four orders of magnitude, but taking all the \emph{a posteriori} information into account clearly makes the biggest difference. In doing so, the bound comes close to the numerical data up to a factor of 6-11 (see \Cref{tab:error-bound-deviations}) and can even resolve the aliasing artifacts.

Let us briefly discuss the visible oscillations in the actual errors. We consider the case d), as the effect is most prominent here. The optimal solution is to quickly switch to a mostly horizontal trajectory in the middle between the first and second row of vortices and to switch back very late, using the spin of both the very first and the very last vortex. Since the horizontal part of the trajectory amounts to the majority of the travel time, it is crucial to hit the right level between the two rows.

Graph-based shortest paths, which, unsurprisingly, tend to mimic this strategy, are, however, restricted to certain discrete levels. Consequently, the error is sensitive to the exact node positions.
If the optimal level is matched by a row of nodes, the error will attain a minimum. On the other hand, if the nodes are positioned such that the optimal trajectory lies exactly between two rows of nodes, we see a maximum error.
Obviously, these are nothing more than local deviations from a clear trend.

Finally, it is interesting to notice that in all four test cases the angular error term of the \emph{a posteriori} bound $\int\alpha_2 \ndxt^2 d\tau$ would have been enough to bound the numerical data alone, which lets us conclude that, even though the bound is sharp in the worst case, in particular the average angular error is not perfectly captured.

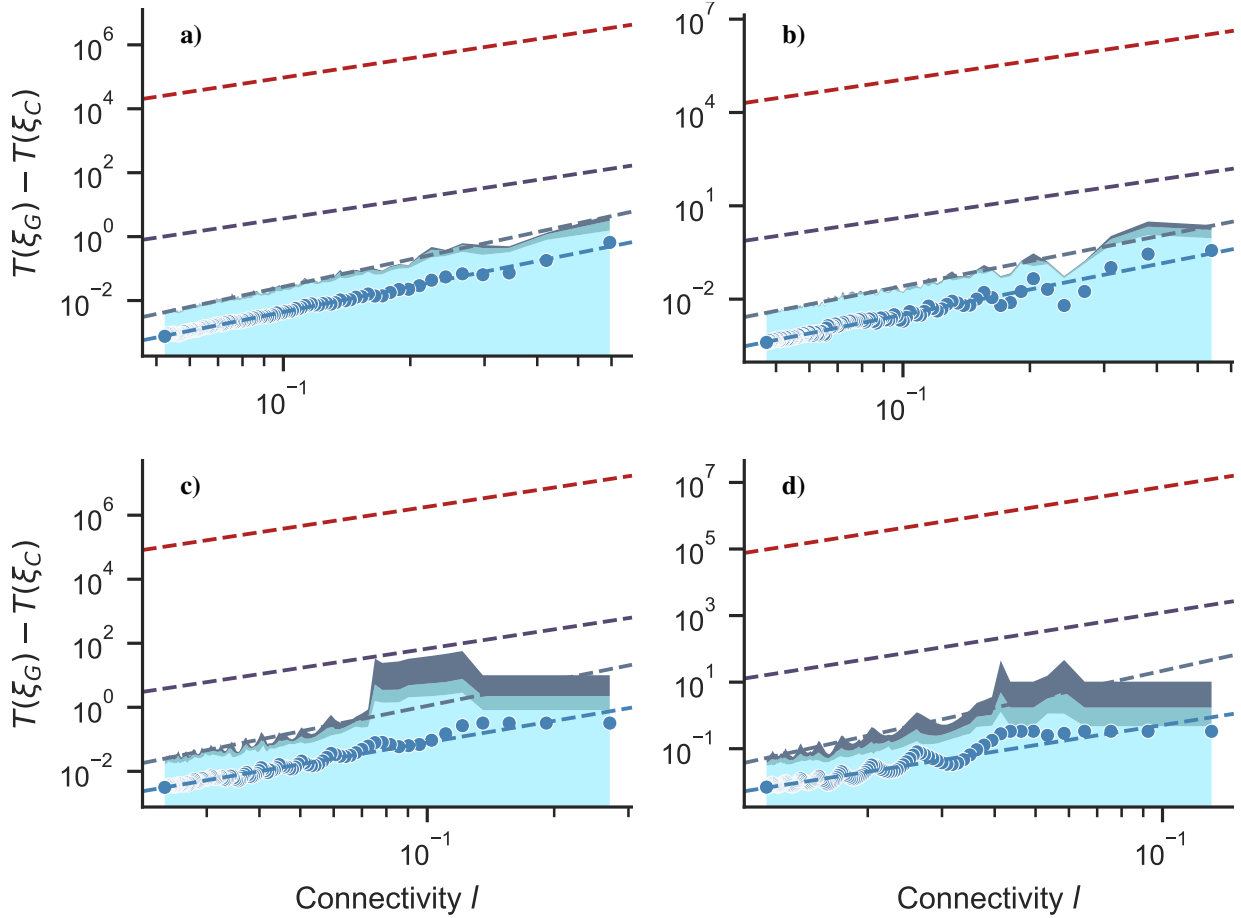
\begin{figure}[!ht]
	\begin{tikzpicture}[
			img/.style={inner sep=0pt},
			beam/.style={draw=white,
				fill=white},
			txt/.style={
				%fill=white,
				%fill opacity=0.8,
				%draw opacity=1,
				%text opacity=1,
				anchor=west,
				font=\bfseries}
		]
		%	\node[img] (a) at (0,4.2)
		%	{\includegraphics[width=.5\textwidth]{demo_1/error_bound/paper_error_bound_log_various.pdf}};
		%	% white beam over x-label
		%	\draw[beam] (-2,1.2) rectangle ++(5,1.1);

		\node[img] (b) at (8,0)
		{\includegraphics[width=.5\textwidth]{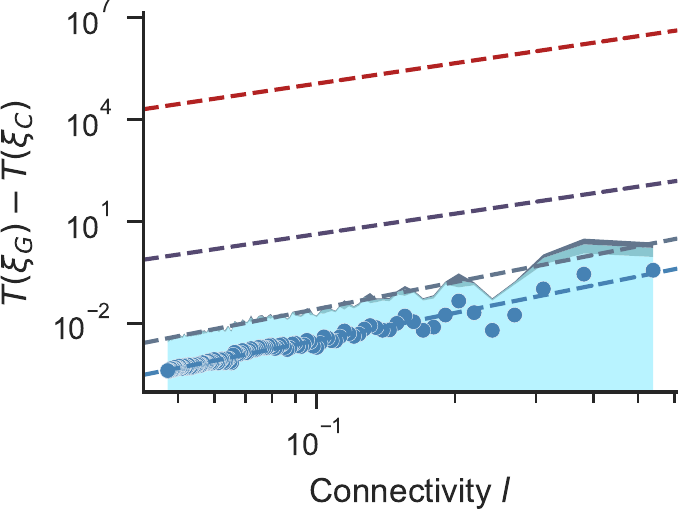}};
		% White beam over x-label
		\draw[beam] (5.5,-3.6) rectangle ++(5,1.1);
		% white beam over y-label
		\draw[beam] (3.2,-1.5) rectangle ++(1.2,3.7);

		\node[img] (a) at (0,0)
		{\includegraphics[width=.5\textwidth]{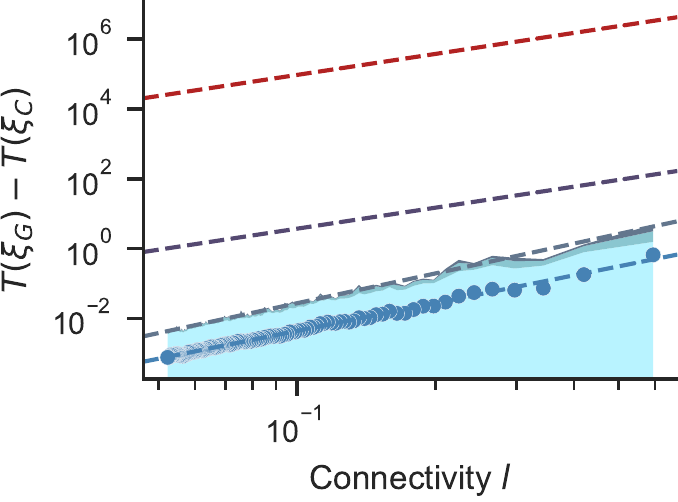}};
		% white beam over x-label
		\draw[beam] (-2.5,-3.5) rectangle ++(5,1.1);

		\node[img] (d) at (8,-6)
		{\includegraphics[width=.5\textwidth]{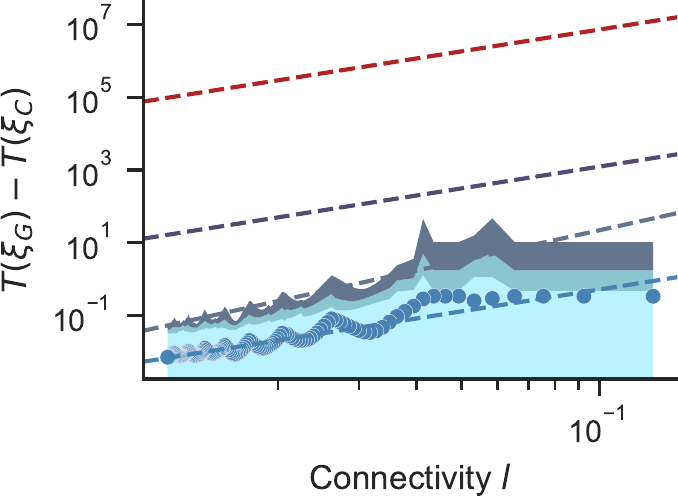}};
		% white beam over y-label
		\draw[beam] (3.2,-7) rectangle ++(1.2,3.7);

		\node[img] (c) at (0,-6)
		{\includegraphics[width=.5\textwidth]{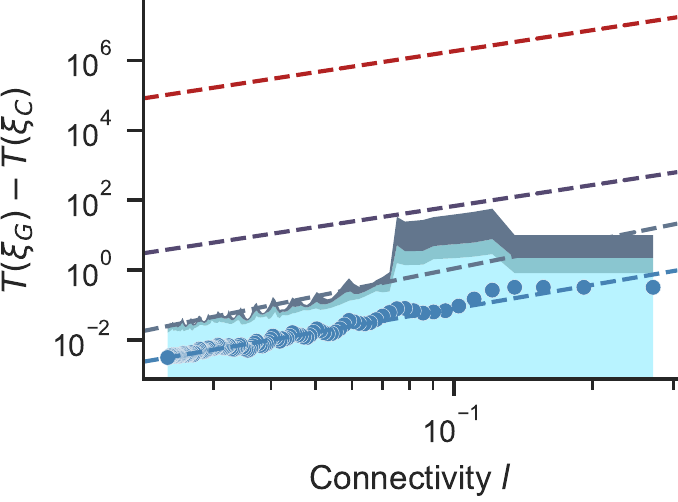}};

		\node[txt](a_text) at (-2, 2.7) {a)};
		\node[txt](b_text) at ( 6, 2.7) {b)};
		\node[txt](c_text) at (-2,-3.3) {c)};
		\node[txt](d_text) at ( 6,-3.3) {d)};
	\end{tikzpicture}
	\caption{Evaluation of the derived error bounds iii) \emph{a priori} (top, red), ii) \emph{local} (middle, purple), and i) \emph{a posteriori} (bottom, gray), comprising three terms visualized by the colored areas. From top to bottom: $\int\alpha_0 \ndx^2 d\tau$, $\int\alpha_1 \ndx \ndxt d\tau$, $\int\alpha_2 \ndxt^2 d\tau$. The blue dots represent results from numerical experiments. The subfigures a)-d) refer to the corresponding test instances.
	}
	\label{fig:error-bound-three-ways}
\end{figure}

\begin{table}[!ht]
\begin{center}
\begin{minipage}{0.5\textwidth}
	\caption{Exponent $p$ for the fitted trend lines $l^p$ of the \emph{a posteriori} bound i) and the numerical results for all four test instances a)-d) as depicted in \Cref{fig:error-bound-three-ways}.
		The \emph{a priori} and \emph{local} bounds are in $\mathcal{O}(l^2)$ by definition.}
	\label{tab:error-bound-trends}
	\begin{center}
	\begin{tabular}{@{}lllll@{}}
		\toprule
		                       & a)   & b)   & c)   & d)   \\ \midrule
		i) \emph{a posteriori} & 2.84 & 2.65 & 2.65 & 2.79 \\
		numerical              & 2.64 & 2.69 & 2.25 & 2.00 \\
		\toprule               &
	\end{tabular}
	\end{center}
\end{minipage}
\end{center}
\end{table}

\begin{table}[!ht]
\begin{center}
\begin{minipage}{0.67\textwidth}
	\caption{Average ratio between the bounds on $T(\xi_G)-T(\xi_C)$ and the actual differences for all four test cases a)-d).}
	\label{tab:error-bound-deviations}
	\begin{center}
	\begin{tabular}{@{}lllll@{}}
		\toprule
		                       & a)               & b)               & c)              & d)              \\ \midrule
		iii) \emph{a priori}   & $2.6 \cdot 10^7$ & $4.9 \cdot 10^7$ & $3.1\cdot 10^7$ & $1.5\cdot 10^7$ \\
		ii) \emph{local}       & $1.0 \cdot 10^3$ & $1.8 \cdot 10^3$ & $1.1\cdot 10^3$ & $2.5\cdot 10^3$ \\
		i) \emph{a posteriori} & $5.9$            & $8.5$            & $9.3$           & $1.1\cdot 10^1$ \\
		\toprule               &
	\end{tabular}
	\end{center}
\end{minipage}
\end{center}
\end{table}

\pagebreak
\begin{figure}%[!ht]
	\begin{tikzpicture}[
			img/.style={inner sep=0pt},
			beam/.style={draw=white, fill=white},
			txt/.style={text=black,font=\bfseries}
		]

		\node[img] (b) at (8,0)
		{\includegraphics[width=.43\textwidth]{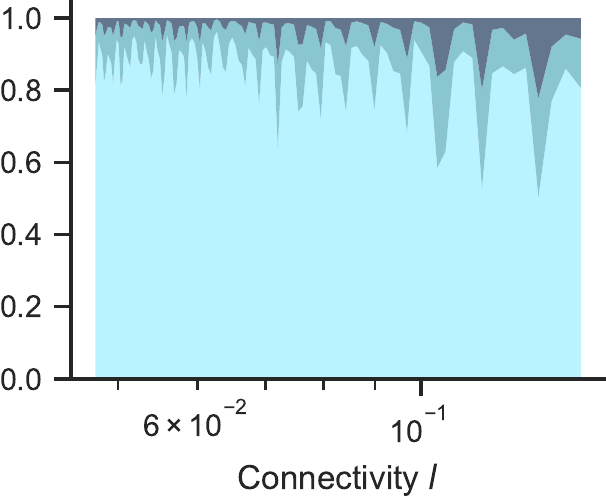}};
		% White beam over x-label
		\draw[beam] (5.5,-3.6) rectangle ++(5,1.1);
		% white beam over y-label
		\draw[beam] (3.2,-1.5) rectangle ++(1.2,3.7);

		\node[img] (a) at (0,0)
		{\includegraphics[width=.43\textwidth]{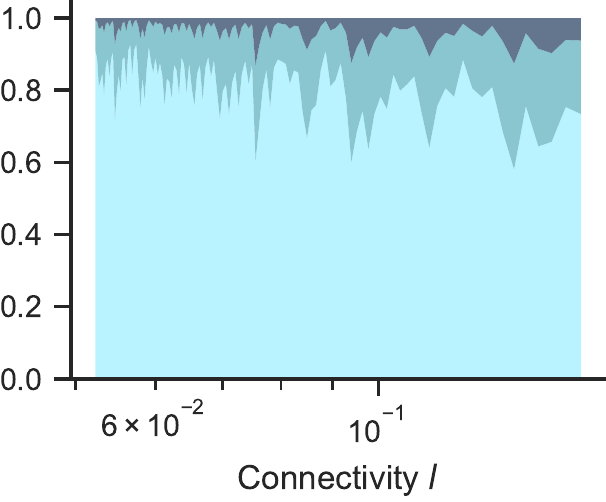}};
		% white beam over x-label
		\draw[beam] (-2.5,-3.5) rectangle ++(5,1.1);

		\node[img] (d) at (8,-6)
		{\includegraphics[width=.43\textwidth]{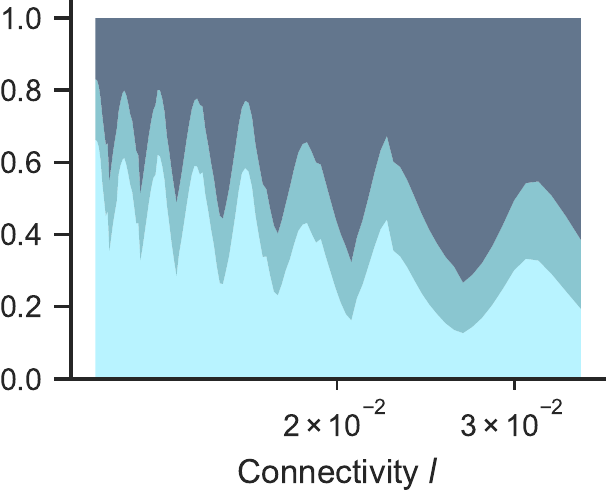}};
		% white beam over y-label
		\draw[beam] (3.2,-7) rectangle ++(1.2,3.7);

		\node[img] (c) at (0,-6)
		{\includegraphics[width=.43\textwidth]{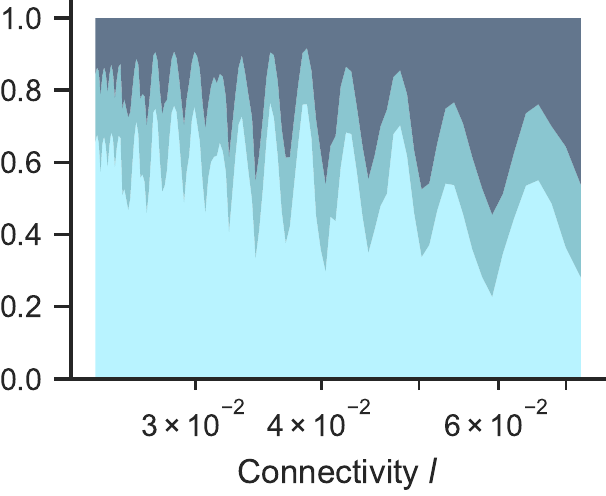}};

		\node[txt](a_text) at (-2, -1) {a)};
		\node[txt](b_text) at ( 6, -1) {b)};
		\node[txt](c_text) at (-2, -7) {c)};
		\node[txt](d_text) at ( 6, -7) {d)};
	\end{tikzpicture}
	\caption{Shares of the three parts of the \emph{a posteriori} error bound; from top to bottom: $\int\alpha_0 \ndx^2 d\tau$, $\int\alpha_1 \ndx \ndxt d\tau$, $\int\alpha_2 \ndxt^2 d\tau$. The subfigures a)-d) refer to the corresponding test instances.}
	\label{fig:integral-error-bound-shares}
\end{figure}
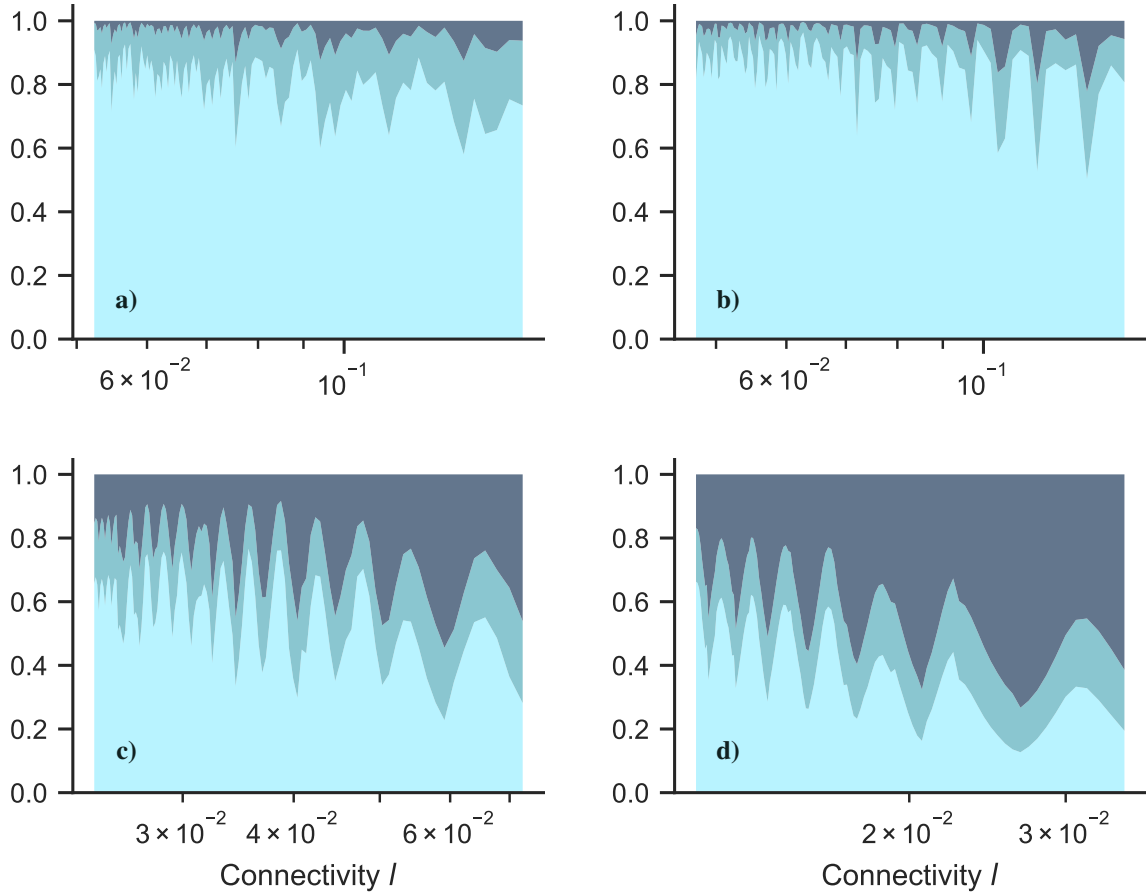

\section{Conclusion}

Discretizing the Zermelo navigation problem with a graph-based approach for computing global optima inevitably leads to approximation errors depending on the graph as well as the continuous optimal path. For a certain class of locally densely connected graphs we have derived three bounds on the excess flight duration in terms of graph and wind properties.

While the \emph{local} bound improves on the \emph{a priori} bound by four orders of magnitude, stressing the importance of using localized quantities if possible, it still is far from sharp in numerical examples. The -- computationally in general unavailable -- \emph{a posteriori} bound, in contrast, is quite sharp, and thus indicates that the use of a posteriori error estimators providing rough approximations of the actual path error $\delta \xi$ can be expected to improve the bounds further. The observed convergence rates, however, agree well with the computational bounds in both cases.

The error bounds derived here can on the one hand guide the choice of optimal graph structures -- the dependence of vertex density $h$ to connectivity length $l$ as presented here is one example --, and on the other hand help identifying switchover points in hybrid discrete-continuous optimization algorithms~\cite{BorndoerferDaneckerWeiser2021}.

\paragraph{Acknowledgments}
We remember Peter Deuflhard for having inspired this project.

\section*{Declarations}

\paragraph{Funding}
	This research was funded by the DFG Research Center of Excellence MATH$^+$ -- Berlin Mathematics Research Center, Project AA3-3.

\paragraph{Competing interests}
	The authors have no relevant financial or non-financial interests to disclose.

\paragraph{Consent for publication}
	We confirm that all authors agree with the submission of this manuscript to Springer Journal of Scientific Computing.

\paragraph{Availability of data and materials}
	The datasets generated during the current study are available from the corresponding author on request.

\paragraph{Code availability}
	The code generated during the current study are available from the corresponding author on request.

\paragraph{Authors' contributions}
	Author Contributions:
	Conceptualization, R.B and M.W.;
	methodology, M.W.;
	software, F.D.;
	validation, F.D.;
	formal analysis, M.W..;
	investigation, F.D. and M.W.;
	resources, R.B., F.D. and M.W.;
	data curation, F.D.;
	writing-original draft preparation, F.D. and M.W.;
	writing-review and editing, R.B.;
	visualization, F.D.;
	supervision, R.B.;
	project administration, R.B. and M.W.;
	funding acquisition, R.B. and M.W.
	All authors have read and agreed to the published version of the manuscript.

\appendix

\section{Supplementary material} \label{sec:supplementary}

Recall from~\eqref{eq:dt-dtau} the derivative
\[
	f(\xi,\xt) = \frac{-\xt^Tw(\xi) + \sqrt{(\xt^Tw(\xi))^2+(\ov^2 - \|w(\xi)\|^2)\xt^2}}{\ov^2 - \|w(\xi)\|^2}
\]
of the time parametrization $t(\tau)$. Here, we will compute and bound its second derivative with respect to $\xi$ and $\xt$ in terms of the wind $w$ and its derivatives.

\begin{theorem} \label{th:hessian-bound}
	Let $c_0(\xi)=\|w(\xi)\| <\ov$, $c_1(\xi)=\|w_x(\xi)\|$, and $c_2(\xi)=\|w_{xx}(\xi)\|$. Moreover, let $L=\xt>0$. Then, the second directional derivative of $f$ is bounded by
	\[
	f(\xi,\xt)'' [\dx,\dxt]^2 \le\alpha_0(\xi) \ndx^2 + \alpha_1(\xi)\ndx \ndxt + \alpha_2(\xi) \ndxt^2
	\]
	with
	\begin{align*}
		\underbar{v}^2 &= \ov^2 - c_0^2, \\
		\alpha_0 &= L \Bigg[
			\frac{c_1^2}{\uv^3} \left(
				1
				+6\frac{c_0}{\uv}
				+2\frac{\sqrt{\ov^2 + c_0^2}}{\uv}
				+6\frac{c_0^2}{\uv^2}
				+8\frac{c_0^3}{\uv^3}
				+8\frac{c_0^2 \sqrt{\ov^2 + c_0^2}}{\uv^3}
			\right)
			\\
		&\hspace{1cm}	+ \frac{c_2}{\uv^2} \left(
				1
				+2\frac{c_0}{\uv}
				+2\frac{c_0^2}{\uv^2}
				+2\frac{c_0 \sqrt{\ov^2 + c_0^2}}{\uv^2}
			\right)
		\Bigg],
		\\
		\alpha_1 &= \frac{c_1}{\uv^2} \left[
			2
			+8\frac{c_0}{\uv}
			+4\frac{c_0^2}{\uv^2}
			+8\frac{c_0^3}{\uv^3}
		\right],
		\\
		\alpha_2 &= \frac1{\uv L} \left[
			1 + 3\frac{c_0^2}{\uv^2}
		\right].
	\end{align*}
\end{theorem}

\begin{proof}
	The derivative $f=t_\tau$ of parametrized time  consists of two terms, the tailwind term
	\[
		f_1 = -\frac{\xt^T w}{g}, \quad g = \ov^2 - w^T w,
	\]
	and the length term
	\[
		f_2 =  g^{-1}\left( (\xt^T w)^2 + g (\xt^T\xt)\right)^{1/2}.
	\]
	At each time $\tau$, we obtain
	\[
		\uv^2:=\ov^2 - c_0^2 \le g \le \ov^2.
	\]
	The directional derivatives of $g$ in direction $(\dx,\dxt)$ read
	\begin{align*}
		g' \dx &= -2w^T w_x \dx \quad\Rightarrow\quad \|g'\| \le 2c_0 c_1
	\end{align*}
	and, as we are only interested in second order directional derivatives,
	\begin{align*}
		\dx^T g'' \dx &= - 2\dx^T( w_x^T w_x+w^Tw_{xx}) \dx \quad\Rightarrow\quad \|g''\|\le 2(c_1^2+c_0c_2).
	\end{align*}
	For the tailwind term, we consider
	\[
		f_1' \dx = - g^{-2} \left((\dxt^T w + \xt^Tw_x\dx)g - \xt^Tw g'\dx\right).
	\]
	Again, we are only interested in second directional derivatives and thus consider
	\begin{align*}
		f_1''[\dx,\dx]
		&= -\bigg[-2g^{-3} g'\dx \left((\dxt^T w + \xt^Tw_x\dx)g - \xt^Tw g'\dx\right) \\
		&\qquad + g^{-2} \Big( (\dxt^T w_x \dx + \dx^T(\xt^T w_{xx})\dx + \dxt^T w_x \dx)g \\
		&\qquad\qquad\qquad + (\dxt^T w + \xt^Tw_x\dx) g' \dx - \dxt^T w g' \dx\\
		&\qquad\qquad\qquad  - \xt^T w_x \dx^T g'\dx - \xt^Tw \dx^Tg''\dx \Big) \bigg]
		\\
		&= \dx^T \bigg[ 2g^{-2}w_x^T\xt g' -2g^{-3}g'^T\xt^Twg' -g^{-1}(\xt^T w_{xx}) \\
		&\qquad\qquad - g^{-2}w_x^T\xt g' + g^{-2}w_x^T\xt g'+g^{-2}\xt^T w g'' \bigg] \dx \\
		&\quad + \dxt^T \bigg[2g^{-2}wg' -2g^{-1}w_x - g^{-2}wg' + g^{-2} wg' \bigg]\dx
		\\
		&= \dx^T \bigg[ 2g^{-2}w_x^T\xt g' -2g^{-3}g'^T\xt^Twg' -g^{-1}(\xt^T w_{xx})
		+g^{-2}\xt^T w g'' \bigg] \dx \\
		&\quad + \dxt^T \bigg[2g^{-2}wg' -2g^{-1}w_x \bigg]\dx.
	\end{align*}
%	which yields
%	\begin{align*}
%		|f_1''[\dx,\dx]|
%		&\le \uv^{-3}L\bigg(
%		4 \uv^{-1}c_0 c_1^2
%		+ 8\uv^{-3}c_0^3c_1^2
%		+ \uv c_2
%		+ 2\uv^{-1}c_0(c_1^2+c_0c_2)
%		\bigg) \ndx^2  \\
%		&\quad + 2\uv^{-2} c_1 \left(
%		2\uv^{-2}c_0^2 + 1
%		\right) \ndxt\ndx \,.
%	\end{align*}
	%
	Now we turn to $f_2$, first considering the term  $F := (\xt^T w)^2 + g (\xt^T\xt)$ with
	\[
		\uv^{2} L^2 \le F \le L^2 (\ov^2 + c_0^2).
	\]
	Then,
	\[
		F' \dx = 2\xt^Tw(\dxt^Tw + \xt^Tw_x\dx) + g'\dx \xt^T\xt + 2g\xt^T \dxt
	\]
%	yields
%	\begin{align*}
%		|F'\dx|
%		&\le 2Lc_0^2 \ndxt + 2L^2 c_0c_1\ndx + 2c_0c_1L^2\ndx + 2\ov^2L\ndxt
%		\\
%		&= (2L^2 c_0c_1 + 2c_0c_1L^2)\ndx + (2Lc_0^2  + 2\ov^2L)\ndxt
%		\\
%		&= 4L^2 c_0c_1 \ndx + 2L(c_0^2  + \ov^2)\ndxt \,.
%	\end{align*}
%	Consequently,
	and
	\begin{align*}
		F''[\dx,\dx]
		&= 2(\dxt^Tw + \xt^Tw_x\dx)^2 \\
		&\quad + 2\xt^Tw(\dxt^T w_x \dx + \dxt^T w_x \dx + \dx^T(\xt^T w_{xx})\dx)  \\
		&\quad + \dx^T g'' \dx \xt^T\xt + 2g'\dx \xt^T\dxt
		+ 2g'\dx \xt^T\dxt + 2g\dxt^T\dxt
		\\
		&= \dx^T \bigg[2w_x^T\xt\xt^Tw_x + 2\xt^Tw (\xt^T w_{xx})
		+ \xt^T\xt g'' \bigg] \dx \\
		&\quad + \dxt^T \bigg[4w\xt^Tw_x + 4\xt^Tw w_x
		+ 4\xt g' \bigg]\dx \\
		&\quad + \dxt^T\left[2ww^T + 2g\right]\dxt.
	\end{align*}
%	yields
%	\begin{align*}
%		|F''[\dx,\dx]|
%		&\le 4L^2(c_1^2+c_0c_2 ) \ndx^2 + 16Lc_0c_1 \ndx\ndxt + 2\ov^2 \ndxt^2.
%	\end{align*}
	For $f_2 = g^{-1} \sqrt{F}$, we thus obtain
	\begin{align*}
		f_2' \dx &= -g^{-2} g'\dx F^{1/2} + \frac{1}{2}g^{-1} F^{-1/2} F'\dx.
	\end{align*}
	The second directional derivative is
	\begin{align*}
		f_2''[\dx,\dx]
		&= 2g^{-3} (g'\dx)^{2} F^{1/2} -g^{-2} \dx^T g''\dx F^{1/2} -g^{-2}g'\dx \frac{1}{2}F^{-1/2}F'\dx \\
		&\quad - \frac{1}{2}g^{-2}g'\dx F^{-1/2}F'\dx -\frac{1}{4}g^{-1}F^{-3/2}(F'\dx)^2 \\
		&\quad + \frac{1}{2}g^{-1}F^{-1/2} \dx^TF''\dx
		\\
		&= 2g^{-3} (g'\dx)^{2} F^{1/2} -g^{-2} \dx^T g''\dx F^{1/2} -g^{-2}g'\dx F^{-1/2}F'\dx
		\\
		&\quad  -\frac{1}{4}g^{-1}F^{-3/2}(F'\dx)^2 + \frac{1}{2}g^{-1}F^{-1/2} \dx^TF''\dx .
	\end{align*}
%	which is bounded by
%	\begin{align*}
%		|f_2''[\dx,\dx]|
%		&\le 2\uv^{-6} 4c_0^2c_1^2 L\sqrt{c_0^2+\ov^2}\ndx^2
%		+ \uv^{-4}  2(c_1^2+c_0c_2)L\sqrt{c_0^2+\ov^2}\ndx^2 \\
%		&\quad  + \uv^{-4}2c_0c_1 \uv^{-1}L^{-1}\|\delta \xi\| \left(4L^2 c_0c_1 \ndx + 2L(c_0^2  + \ov^2)\ndxt\right) \\
%		&\quad + \frac{1}{4}\uv^{-2}\uv^{-3}L^{-3} \left(4L^2 c_0c_1 \ndx + 2L(c_0^2  + \ov^2)\ndxt\right)^2 \\
%		&\quad + \frac{1}{2}\uv^{-2}\uv^{-1}L^{-1} \big(4L^2(c_1^2+c_0c_2 ) \ndx^2 + 16Lc_0c_1 \ndx\ndxt + 2\ov^2 \ndxt^2\big)
%		\\
%		&= 2\uv^{-4}L\sqrt{c_0^2+\ov^2}\, (4\uv^{-2} c_0^2c_1^2
%		+ c_1^2+c_0c_2)\ndx^2 \\
%		&\quad +   8\uv^{-5}c_0^2c_1^2 L \ndx^2 + 4\uv^{-5}c_0c_1 (c_0^2  + \ov^2)\|\delta \xi\|\ndxt \\
%		&\quad + \uv^{-5}L^{-1} \bigg(4L^2 c_0^2c_1^2 \ndx^2 \\
%		&\quad\quad\quad\quad\quad\quad+4L c_0c_1(c_0^2  + \ov^2)\ndxt \ndx+ (c_0^2  + \ov^2)^2\ndxt^2\bigg) \\
%		&\quad + 2\uv^{-3}L(c_1^2+c_0c_2 ) \ndx^2 +  8\uv^{-3}c_0c_1 \ndx\ndxt + \uv^{-3}L^{-1} \ov^2 \ndxt^2
%		\\
%		&= L \uv^{-3}\bigg[
%			2\uv^{-1}\sqrt{c_0^2+\ov^2}\, (4\uv^{-2} c_0^2c_1^2 + c_1^2 + c_0c_2 ) \\
%			&\quad\quad\quad\quad
%			+ 12\uv^{-2}c_0^2c_1^2
%			+ 2c_1^2
%			+ 2c_0c_2
%		\bigg] \ndx^2 \\
%		&\quad + 2c_1\uv^{-3}\left[
%		4\uv^{-3} c_0 (c_0^2  + \ov^2)
%		+ 4\uv^{-3} c_0
%		\right] \|\delta \xi\|\ndxt \\
%		&\quad + L^{-1}\uv^{-1}\left[
%		\uv^{-4}  (c_0^2  + \ov^2)^2 +  \uv^{-2} \ov^2
%		\right] \ndxt^2.
%	\end{align*}
	Adding $f_1''$ and $f_2''$, we finally obtain
	\begin{align*}
		&\hspace{-1cm} f''(\xi,\xt)[\dx,\dxt]^2 \\
		&= (f_1'' + f_2'')[\dx,\dxt]^2 \\
		&=      - 2g^{-3} (g'\dx)^2 (\xt^Tw)         &&
				+ g^{-2} (\dx^Tg''\dx) (\xt^Tw)      \\
		&\quad	+ 2g^{-2} (g'\dx) (\xt^Tw_x\dx)      &&
				- g^{-1} w_{xx}[\xt,\dx,\dx]         \\
		&\quad	+ 2g^{-2} (g'\dx) (\dxt^T w)         &&
				- 2g^{-1} (\dxt^T w_x \dx),   \\
%%%%%%%%%%%
		&\quad	+2g^{-3} (g'\dx)^2 F^{1/2}
		&&		- g^{-2} (\dx^T g''\dx) F^{1/2} \\
		&\quad	- g^{-2} (g'\dx) F^{-1/2} F'[\dx,\dxt]
		&&		+ \frac12 g^{-1} F^{-1/2} F''[\dx,\dxt]^2  \\
		&\quad	- \frac14 g^{-1} F^{-3/2} (F'[\dx,\dxt])^2,
	\end{align*}
	which is bounded by
%	\begin{align*}
%		f''
%		&\le    +8\uv^{-6} L   c_0^3 c_1^2 \ndx^2
%		&&		+2\uv^{-4} L   c_0 c_1^2 \ndx^2 \\
%		&\quad	+2\uv^{-4} L   c_0^2 c_2 \ndx^2
%		&&		+4\uv^{-4} L   c_0 c_1^2 \ndx^2       \\
%		&\quad	+ \uv^{-2} L   c_2 \ndx^2
%		&&		+4\uv^{-4}     c_0^2 c_1 \ndx \ndxt          \\
%		&\quad	+2\uv^{-2}     c_1 \ndx \ndxt
%	%%%%%%%%%%%
%		&&		+8\uv^{-6} L   c_0^2 c_1^2 \sqrt{\ov^2 + c_0^2} \ndx^2     \\
%		%
%		&\quad	+2\uv^{-4} L   c_1^2 \sqrt{\ov^2 + c_0^2} \ndx^2
%		&&		+2\uv^{-4} L   c_0 c_2 \sqrt{\ov^2 + c_0^2} \ndx^2 \\
%		%
%		&\quad	+4\uv^{-5}     c_0^3 c_1 \ndx \ndxt
%		&&		+4\uv^{-5} L   c_0^2 c_1^2 \ndx^2\\
%		%
%		&\quad  +2\uv^{-3}     c_0 c_1\ndx \ndxt
%		&&	  	+2\uv^{-3}     c_0 c_1\ndx \ndxt \\
%		&\quad  + \uv^{-3} L^{-1} c_0^2 \ndxt^2
%		&&	  	+ \uv^{-3} L   c_1^2 \ndx^2 \\
%		&\quad  + \uv^{-3} L   c_0 c_2 \ndx^2
%		&&  	- \uv^{-3} L   c_0 c_2 \ndx^2 \\
%		&\quad  + \uv^{-1} L^{-1} \ndxt^2
%		%
%		&&		+2\uv^{-5}     c_0^3 c_1 \ndx \ndxt \\
%		&\quad	+2\uv^{-5}     c_0^3 c_1 \ndx \ndxt
%		&&		+2\uv^{-5} L   c_0^2 c_1^2 \ndx^2\\
%		&\quad	+2\uv^{-3} L^{-1} c_0^2 \ndxt^2
%		&&		+2\uv^{-3}     c_0 c_1 \ndx \ndxt\\
%		&\quad	+2\uv^{-3}     c_0 c_1 \ndx \ndxt
%	\end{align*}
%	sorted
	\begin{align*}
		f''
		&\le L \Bigg[
			\frac{c_1^2}{\uv^3} \left(
				1
				+6\frac{c_0}{\uv}
				+2\frac{\sqrt{\ov^2 + c_0^2}}{\uv}
				+6\frac{c_0^2}{\uv^2}
				+8\frac{c_0^3}{\uv^3}
				+8\frac{c_0^2 \sqrt{\ov^2 + c_0^2}}{\uv^3}
			\right)
			\\
		&\hspace{1cm}	+ \frac{c_2}{\uv^2} \left(
				1
				+2\frac{c_0}{\uv}
				+2\frac{c_0^2}{\uv^2}
				+2\frac{c_0 \sqrt{\ov^2 + c_0^2}}{\uv^2}
			\right)
		\Bigg] \ndx^2
		\\
		&\quad+\frac{c_1}{\uv^2} \left[
			2
			+8\frac{c_0}{\uv}
			+4\frac{c_0^2}{\uv^2}
			+8\frac{c_0^3}{\uv^3}
		\right]\ndx \ndxt
		\\
		&\quad+\frac1{\uv L} \left[
			1 + 3\frac{c_0^2}{\uv^2}
		\right] \ndxt^2.
	\end{align*}
\end{proof}

Since the claim of \Cref{th:hessian-bound} is rather unwieldy, we simplify it, finally proving \Cref{lem:second-derivative-bound}.

\setcounter{theorem}{\value{lemma-stored}}
\begin{lemma}
	For any $p\in\Omega$ let $c_0(p)=\|w(p)\|$, $c_1(p)=\|w_x(p)\|$, and $c_2(p)=\|w_{xx}(p)\|$, and assume $c_0 \le \ov / \sqrt{5}$.
	Moreover, let $\xi\in\hat X$, $L:=\xt>0$ and $\underbar{v}^2(p) := \ov^2 - c_0^2(p)$.
	Then
	\begin{align*}
		\alpha_0(p) &\le \frac{L}{\uv^3(p)} \left(
					12 c_1^2(p)
					+ 4 c_2(p) \uv(p)
				\right), \\
		\alpha_1(p) &\le \frac{8c_1(p)}{\uv^2(p)},\\
		\alpha_2(p) &\le \frac2{L\uv(p)}.
	\end{align*}
	hold in \Cref{th:hessian-bound}.
\end{lemma}

\begin{remark}
	The assumption of $c_0\le \ov/\sqrt{5}$ covers the usually experienced wind velocities, but not the possible extremes.
\end{remark}

\begin{proof}
	Let $s:=c_0/\ov$ be the relative wind speed. Then
	\[
		\frac{c_0}{\uv} =  \frac{s\ov}{\sqrt{\ov^2-s^2\ov^2}} = \frac{s}{\sqrt{1-s^2}} \le \frac12,
	\]
	\[
		1 \le \frac{\sqrt{c_0^2+\ov^2}}{\uv} \le \sqrt{\frac{3}{2}},
	\]
	and
	\[
		\frac{\ov}{\uv} = \frac{\ov}{\sqrt{\ov^2 - s^2\ov^2}} = \frac{1}{\sqrt{1-s^2}} \le \frac{\sqrt5}{2},
	\]
	which allows to bound
	\begin{align*}
		a_0 &:= L \Bigg[
			\frac{c_1^2}{\uv^3} \left(
				1
				+6\frac{c_0}{\uv}
				+2\frac{\sqrt{\ov^2 + c_0^2}}{\uv}
				+6\frac{c_0^2}{\uv^2}
				+8\frac{c_0^3}{\uv^3}
				+8\frac{c_0^2 \sqrt{\ov^2 + c_0^2}}{\uv^3}
			\right)
			\\
			&\hspace{1cm}	+ \frac{c_2}{\uv^2} \left(
				1
				+2\frac{c_0}{\uv}
				+2\frac{c_0^2}{\uv^2}
				+2\frac{c_0 \sqrt{\ov^2 + c_0^2}}{\uv^2}
			\right)
		\Bigg]
		\\
		&\quad\le \left[
					\frac{c_1^2}{\uv^3} \left(\frac{13}{2} +4\sqrt{\frac32}\right)
					+ \frac{c_2}{\uv^2} \left(\frac52 +\sqrt{\frac32}\right)
				\right]
		\\
		&\quad\le L \left(
			12 \frac{c_1^2}{\uv^3}
			+ 4 \frac{c_2}{\uv^2}
		\right)
	\end{align*}
	as well as
	\begin{align*}
		a_1 &:= \frac{c_1}{\uv^2} \left[
			2
			+8\frac{c_0}{\uv}
			+4\frac{c_0^2}{\uv^2}
			+8\frac{c_0^3}{\uv^3}
		\right]
		\le \frac{8c_1}{\uv^2}
	\end{align*}
	and
	\begin{align*}
		a_2 &:= \frac1{\uv L} \left[
			1 + 3\frac{c_0^2}{\uv^2}
		\right]
		\le \frac7{4\uv L}
		\le \frac2{L\uv}.
	\end{align*}
\end{proof}

\setcounter{theorem}{\value{lemma-stored-2}}

% \todo[inline]{MW: Knapp zwei Seiten -- das k{\"o}nnte man sogar ins Paper ziehen.}
\begin{lemma}
	Let $\xi_C \in\hat X$ be a minimizer of~\eqref{eq:reduced-problem}. Then, it is twice continuously differentiable and its second derivative is bounded by
	\begin{equation}
		\|(\xi_C)_{\tau\tau}\| \le \os := \frac{\oc_1L^2}{\ov - \oc_0 }\left( \sqrt{2}\,\ov    + \frac{\ov
			+ \oc_0}{\ov - \oc_0}     \left((1+\sqrt{2})\ov  + \oc_0\right)   \right).
	\end{equation}
	For $\oc_0 \le \ov / \sqrt{5}$ this simplifies to $\os \le 17\oc_1L^2$.
\end{lemma}
\begin{proof}
	The optimal control problem~\eqref{eq:continuous-problem} has originally been formulated by Zermelo~\cite{Zermelo1931} in terms of the heading angle $\varphi$ in unscaled time $t$ instead of the airspeed $v$ in scaled time $\tau$, which are related by
	\begin{equation} \label{eq:airspeed-angle-relation}
		v(\tau) = \ov \begin{bmatrix} \cos\varphi(T\tau) \\ \sin\varphi(T\tau) \end{bmatrix}.
	\end{equation}
	The Hamiltonian formalism yields an expression for the heading angle rate of an optimal trajectory,
	\begin{equation*}
		\varphi_t = w_x : B, \quad B:=\begin{bmatrix} \cos\varphi \sin\varphi & - \cos^2 \varphi \\
		\sin^2\varphi & - \cos\varphi\sin\varphi \end{bmatrix},
	\end{equation*}
	with ``$:$'' denoting tensor contraction, and confirms the regularity of $\xi$. By the chain rule,~\eqref{eq:airspeed-angle-relation} yields
	\[
		v_\tau = \ov \begin{bmatrix} -\sin\varphi \\ \cos\varphi\end{bmatrix} \varphi_t T
	\]
	and a bound
	\[
		\|v_\tau\| = \ov T |\varphi_t| \le \ov T \|w_x\|_F \|B\|_F \le \sqrt{2}\,\ov \oc_1 T
	\]
	using the Frobenius norm $\|\cdot\|_F$.
	For the ground speed $\nu(\tau) := v(\tau) + w(x(\tau))$ we thus obtain $ \|\nu\| \le \ov + \oc_0$ and
	\begin{align*}
		\|\nu_\tau\| & \le \|v_\tau + w_x x_\tau\| \le \oc_1 T \left((1+\sqrt{2})\ov  + \oc_0\right).
	\end{align*}

	The flight path $\xi_C$ (omitting the subscript $C$ in the following) with constant ground speed $\xt = L$ is related to the actual flight path $x$ by $\xi(\tau) = x(r(\tau))$ with $r:[0,1]\to[0,1]$ being a monotone bijection. Therefore,
	\begin{align*}
		\xt &= x_\tau(r)r_\tau = T\nu(r)  r_\tau
	\end{align*}
	yields
	\[
		L = \xt \le T\|\nu\| r_\tau \le T(\ov + \oc_0) r_\tau
		\quad\Rightarrow\quad
		r_\tau \ge \frac{L}{T(\ov + \oc_0)}
	\]
	and similarly $r_\tau \le L/(T(\ov - \oc_0))$.

	For the  second derivative, we note that
	\[
		0 = (L^2)_\tau = (\xt^T \xt)_\tau = 2 \xi_{\tau\tau}^T \xt,
	\]
	which means that the curvature vector $\xi_{\tau\tau}$ is orthogonal to the path and the ground velocity $\nu$. Consequently, we obtain
	\begin{align*}
		0 &= (\xt^T \nu(r) r_\tau)_\tau
		= \underbrace{\xi_{\tau\tau}^T \nu}_{=0} r_\tau + \xt^T \nu_\tau r_\tau^2 + \xt \nu r_{\tau\tau}
		= \xt^T\nu_\tau r_\tau^2 + \frac{L^2}{T r_\tau} r_{\tau\tau}
	\end{align*}
	and therefore
	\begin{align*}
		|r_{\tau\tau}|
		&\le \frac{Tr_\tau}{L^2} \xt \|\nu_\tau\| r_\tau^2
		\le\frac{\oc_1 T^2r_\tau^3}{L} \left((1+\sqrt{2})\ov  + \oc_0\right).
	\end{align*}
	Now we can bound
	\begin{align*}
		\|\xi_{\tau\tau}\|
		&\le T \| \nu_\tau \| r_\tau^2 + T \|\nu\| |r_{\tau\tau}|  \\
		&\le T\left( \sqrt{2}\,\ov \oc_1 T r_\tau^2 + (\ov
		+ \oc_0)  \frac{\oc_1 T^2r_\tau^3}{L}     \left((1+\sqrt{2})\ov  + \oc_0\right)   \right) \\
		&\le \oc_1T^2 r_\tau^2\left( \sqrt{2}\,\ov    + (\ov
		+ \oc_0)  \frac{ Tr_\tau}{L}     \left((1+\sqrt{2})\ov  + \oc_0\right)   \right) \\
		&\le \frac{\oc_1L^2}{\ov - \oc_0 }\left( \sqrt{2}\,\ov    + \frac{\ov
			+ \oc_0}{\ov - \oc_0}     \left((1+\sqrt{2})\ov  + \oc_0\right)   \right),
	\end{align*}
	which completes the proof.
\end{proof}

% ------------------------------------------------------------------------- %
%             Bib                                                           %
% ------------------------------------------------------------------------- %

\printbibliography

\end{document}